\documentclass[reqno]{amsart}
\usepackage{txfonts}
\usepackage{amsfonts}
\usepackage{amssymb}
\usepackage[mathscr]{eucal}
\usepackage{amsmath}
\usepackage[bookmarksopen,colorlinks,citecolor=blue,
linkcolor=black,pdfstartview=FitH]{hyperref}
\usepackage{ytableau}
\usepackage{float}
\usepackage{url}
\usepackage{appendix}
\usepackage{graphicx}

\newtheorem{theorem}{Theorem}[section]
\newtheorem{corollary}[theorem]{Corollary}
\newtheorem{lemma}[theorem]{Lemma}
\newtheorem{proposition}[theorem]{Proposition}
\theoremstyle{definition}
\newtheorem{definition}[theorem]{Definition}
\newtheorem{notation}[theorem]{Notation}
\newtheorem{example}[theorem]{Example}
\newtheorem{problem}[theorem]{Problem}

\newtheorem{conjecture}{Conjecture}

\numberwithin{equation}{section}
\theoremstyle{remark}
\newtheorem{remark}[theorem]{Remark}

\allowdisplaybreaks[4]

\makeatletter
\@namedef{subjclassname@2020}{\textup{2020} MSC}
\makeatother

\begin{document}

\title{Two classes of minimal generic fundamental invariants for tensors}
\pdfbookmark[0]{Two classes of minimal generic fundamental invariants for tensors}{}

\author[li]{Xin Li}
\address{School of Mathematical Sciences, Zhejiang University of Technology, Hangzhou 310023, P. R. China}
\email{xinli1019@126.com (Xin Li)}

\author[zhang]{Liping Zhang}
\address{School of Mathematical Sciences, Qufu Normal University, Qufu 273165, P. R. China}
\email{zhanglp06@gmail.com (Liping Zhang)}

\author[xia]{Hanchen Xia}
\address{School of Mathematical Sciences, Shanghai Jiao Tong University, Shanghai 200240, P. R. China}
\email{x\_hc\_2000@sjtu.edu.cn (Hanchen Xia)}

%
%
\thanks{Xin Li's research is supported by National Natural Science Foundation of China (Grant No.11801506). Liping Zhang's research is supported by the Shandong Provincial Natural Science Foundation (Project ZR2024QA204).}


\begin{abstract}
Motivated by the problems raised by B\"{u}rgisser and Ikenmeyer in \cite{BI17}, we discuss two classes of minimal generic fundamental invariants for tensors of order 3. The first one is defined on $\otimes^3 \mathbb{C}^m$, where $m=n^2-1$. We study its construction by obstruction design introduced by B\"{u}rgisser and Ikenmeyer, which partially answers one problem raised by them.
The second one is defined on $\mathbb{C}^{\ell m}\otimes \mathbb{C}^{mn}\otimes \mathbb{C}^{n\ell}$. We study its evaluation on the matrix multiplication tensor $\langle\ell,m,n\rangle$ and unit tensor $\langle n^2 \rangle$ when $\ell=m=n$. The evaluation on the unit tensor leads to the definition of Latin cube and 3-dimensional Alon-Tarsi problem. We generalize some results on Latin square to Latin cube, which enrich the understanding of 3-dimensional Alon-Tarsi problem. It is also natural to generalize the constructions to tensors of other orders. We illustrate the distinction between even and odd dimensional generalizations by  concrete examples. Finally, some open problems in related fields are raised.
\end{abstract}
\maketitle

\section{Introduction}\label{se:intro}

In this paper, we continue the discussions in second part of \cite{BI17} that focus on the invariants of tensors of order 3.
Motivated by border rank and complexity of matrix multiplication, B\"{u}rgisser and Ikenmeyer provided general properties of fundamental invariants of $\otimes^3\mathbb{C}^m$.
In particular, when $m=n^2$ they gave an explicit description of the minimal generic fundamental invariant of $\otimes^3\mathbb{C}^m$ which is denoted by  $F_n$. That is, $F_n:\otimes^3\mathbb{C}^{n^2}\to \mathbb{C}$ is an $\operatorname{SL}_{n^2}^3$-invariant with minimal degree. In fact, $F_n$ has been appeared in Example 4.12 of \cite{BI13} and Claim 7.2.17 of \cite{Iken12}, where its construction is described by the \emph{obstruction design}. Let $\langle n^2 \rangle$ and $\langle n, n, n\rangle\in\otimes^3\mathbb{C}^{n^2}$ be the unit tensor and matrix multiplication tensor, respectively.  B\"{u}rgisser and Ikenmeyer described the evaluations $F_n(\langle n^2 \rangle)$ and $F_n(\langle n, n, n\rangle)$ by combinatorial conditions.
Interestingly, determining  $F_n(\langle n^2 \rangle)\neq0$ is equivalent to a 3D version of the Alon-Tarsi conjecture, which is called 3-dimensional Alon-Tarsi problem  or 3-dimensional Alon-Tarsi conjecture. B\"{u}rgisser and Ikenmeyer raised several questions on $F_n$. In this paper, our discussions are related to the following three questions:
\begin{problem}\cite[Problem 5.15]{BI17}\label{prob-fnn0}
Give a direct proof of $F_n\neq0$ by evaluating $F_n$ at a (generic) $w\in \otimes^3\mathbb{C}^{n^2}$.
\end{problem}

\begin{problem}\cite[Problem 5.17]{BI17}\label{prob-detlm}
Let $\delta(m)$ denote the generic minimal exponent of tensors of $\otimes^3\mathbb{C}^{m}$. How close is $\delta(m)$ to $\sqrt{m}$ if $m$ is not a square?
\end{problem}

\begin{problem}\cite[Problem 5.23]{BI17}\label{prob-lccref}
Let $n$ be even. Is the number of even Latin cubes of size $n$ different from the number of odd Latin cubes of size $n$?
\end{problem}
The complexity of evaluating highest weight vectors in the polynomial setting was discussed in \cite{BDI20}. From recent researches \cite{BGO+17,BFG+19,BLM+21,GIM+20}, we know that the study of invariants and their evaluations is the key to understand the orbit problems. The earlier study of invariants and their applications in tensor setting was given in \cite{AJR13+}. Problem \ref{prob-lccref} above is a 3-dimensional generalization of the well-known Alon-Tarsi conjecture \cite{AT92,Jan95,DSS12}, see also Problem \ref{prob-lcc} below. Recently, Problem 5.20 of \cite{BI17} and its higher dimensional generalization were discussed in \cite{AY22}.

Focusing on three questions above, we study two classes of  generic fundamental invariants with minimal degree. Suppose that $m=n^2-1$. The first minimal one is an $\operatorname{SL}_m^3$-invariant $F_m:\otimes^3 \mathbb{C}^m\to \mathbb{C}$. Suppose that $B(n,n,n)=\{(i,j,k)|i,j,k\in [n]\}\subseteq \mathbb{Z}^3_{+}$ is the 3-dimensional cube of side length $n$.
Just like Example 4.12 of \cite{BI13}, the construction of $F_m$ is described by an obstruction design, which is obtained by deleting a \emph{diagonal} (see Definition \ref{def-diag}) of $B(n,n,n)$. Determining $F_m\neq0$ is obtained by showing the Kronecker coefficient $k(m\times n,m\times n,m\times n)=1$ (see Theorem \ref{thm-deln}).
So $F_m\neq0$ is also obtained indirectly. Moreover, we get
that  $\delta(n^2-1)=\lceil \sqrt{n^2-1}\rceil=n$, which partially answers Problem \ref{prob-detlm} above. Just like $\langle n, n, n\rangle\in\otimes^3\mathbb{C}^{n^2}$, some interesting tensors also live in  $\otimes^3\mathbb{C}^{n^2-1}$, for example, the structure tensor of $\mathfrak{sl}_n$ studied in \cite{Bari21}.

The second minimal one is a noncubic generalization of $F_n$, which has been stated in Remark 5.14 of \cite{BI17}. Suppose that $\ell$, $m$, $n\in \mathbb{N}$. It is an $\operatorname{SL}_{\ell m}\times \operatorname{SL}_{mn}\times \operatorname{SL}_{n\ell }$-invariant defined on $\mathbb{C}^{\ell m}\otimes \mathbb{C}^{mn}\otimes \mathbb{C}^{n\ell}$, which is denoted by $F_{(\ell,m,n)}$ (see Theorem \ref{thm-lmn}). So when $\ell=m=n$, $F_{(\ell,m,n)}=F_n$.
Therefore, Problem \ref{prob-fnn0} has a generalization as follows
\begin{problem}\label{prob-flmnn0-intr}
Give a direct proof of $F_{(\ell,m,n)}\neq0$ by evaluating $F_{(\ell,m,n)}\neq0$  at a (generic) $w\in \mathbb{C}^{\ell m} \otimes \mathbb{C}^{mn}\otimes\mathbb{C}^{n\ell}$.
\end{problem}
Let $\langle\ell,m,n\rangle\in\mathbb{C}^{\ell m}\otimes \mathbb{C}^{mn}\otimes \mathbb{C}^{n\ell}$ be the (rectangular) matrix multiplication tensor. Let $\langle n^2\rangle\in\otimes^3\mathbb{C}^{n^2}$ be the unit tensor.
We study the evaluations: $F_{(\ell,m,n)}(\langle\ell,m,n\rangle)$ and $F_{n}(\langle n^2\rangle)$. Our results improve the description of $F_{n}(\langle n,n,n\rangle)$ and $F_{n}(\langle n^2\rangle)$ given in \cite{BI17}.

In Section \ref{subsec-evalflmn}, we make a detailed description of the evaluation $F_{(\ell,m,n)}(\langle\ell,m,n\rangle)$. We can see that it is natural to show $F_{(\ell,m,n)}\neq0$ directly by evaluating at $\langle\ell,m,n\rangle$, although $\langle\ell,m,n\rangle$ is not generic generally. By our description, we can compute $F_{(\ell,m,n)}(\langle\ell,m,n\rangle)$ by computer when $\ell,~m,~n$ are small. For example, we verified that $F_{(2,2,2)}(\langle2,2,2\rangle)$ and $F_{(2,2,3)}(\langle2,2,3\rangle)\neq0$ directly (see Example \ref{exm-223} and \ref{exm-222}). In Section \ref{subsec-valfnn2}, we make a more detailed description of $F_{n}(\langle n^2\rangle)$, which leads to the definition of Latin cube and 3-dimensional Alon-Tarsi problem \cite{BI17}. Many results for Alon-Tarsi conjecture can be generalized to 3-dimensional Alon-Tarsi problem. In Section \ref{se:clc}, the results in Section 2 of \cite{DSS12} are generalized to Latin cubes. For example, by introducing 3-dimensional determinant and permanent, the formulae for the enumeration of Latin squares is generalized to Latin cubes. Similar to Latin squares, we can also define symbol-even and symbol-odd Latin cubes. Their difference can also be expressed as the coefficient of some polynomials (see Theorem \ref{thm-seso}).
Just like the discussion in Section 3 of \cite{DSS12}, by our results, in the future we hope the 3-dimensional Alon-Tarsi problem can be solved when $n=p-1$ for the odd prime $p$.

It is also interesting to generalize all results above to tensors of order $k$. That is, construct polynomials on $V=\mathbb{C}^{n_1}\otimes \mathbb{C}^{n_2} \cdots\otimes\mathbb{C}^{n_k}$ that are invariant under the action of $\operatorname{SL}_{n_1}\times \operatorname{SL}_{n_2} \cdots \times \operatorname{SL}_{n_k}$. Obstruction design (3-dimensional) is an important combinatorial tool to get the construction of invariants for tensors of order 3, which was defined by Ikenmeyer et al. in \cite{Iken12,BI13,BlaI18}. So to describe the invariants for tensors of order $k$, the definition of obstruction design should be generalized to $k$-dimensional cases. One forbidden pattern was given for (3-dimensional) obstruction designs in Section 7.2 of \cite{Iken12}. It is not hard to see that the forbidden pattern is not necessary for even dimensional generalization. In Section \ref{se-evenodd}, we give some concrete examples to see the distinction between odd and even generalizations. Our examples are based on the $k$-dimensional hypercube of side length $n$: $B^{(k)}(n)=\{(i_1,i_2,...,i_k)|i_1,i_2,...,i_k\in [n]\}\subseteq \mathbb{Z}^{k}_+$. In Section \ref{se:final}, some remarks and open problems in related fields are given.

The paper is organized as follows. In Section \ref{se:prelim}, the basic definitions and properties we need are summarized. In Section \ref{se:od3}, we construct the minimal generic fundamental invariant of $\otimes^3\mathbb{C}^{n^2-1}$. In Section \ref{sec-lmnn2},  we discuss the evaluations  $F_{(\ell,m,n)}(\langle \ell,m,n \rangle)$ and $F_{n}(\langle n^2 \rangle)$. Specially, we verified that $F_{2}(\langle 2,2,2 \rangle)$ and $F_{(2,2,3)}(\langle 2,2,3 \rangle)\neq0$. For $F_{n}(\langle n^2 \rangle)$, we give a more detailed description of the definition of Latin cube, which leads to the 3-dimensional Alon-Tarsi problem.
In Section \ref{se:clc}, we continue the discussion of the 3-dimensional Alon-Tarsi problem. Results in Section 2 of \cite{DSS12} are generalized to Latin cubes. In Section \ref{se-evenodd}, using the $k$-dimensional hypercube, we discuss the distinction between the odd and even dimensional generalizations. Some remarks and open problems are given in Section \ref{se:final}.

\section{Preliminaries: Highest weight vector, obstruction design and invariant}\label{se:prelim}

In this section, we briefly summarize the results about highest weight vector, obstruction design and invariant in tensor setting.
Most of our notations are borrowed from \cite{Iken12}, \cite{BI13}, \cite{BlaI18} and \cite{BI17}.

Let $\lambda,\mu,\nu$ be partitions of $d$ with at most $m$ rows.
Let ${}^{t}\lambda$ denote the conjugate of $\lambda$.
Let $\operatorname{GL}_m$ and $\operatorname{SL}_m$ denote the general linear group and special linear group, respectively.
An irreducible representation of $\operatorname{GL}^3_m:=\operatorname{GL}_m \times \operatorname{GL}_m \times \operatorname{GL}_m$ (resp. $\operatorname{GL}_m$) is determined by its highest weight vector, which can be indexed by a partition triple $(\lambda,\mu,\nu)$ (resp. a partition $\lambda$).
To describe the evaluation of the highest weight vector, it will be convenient to use the Dirac's bra-ket notation (see e.g. Section 2.5 of \cite{Iken12}). Suppose that $V$ is a finite dimensional complex vector space with Hermitian inner product and $V^*$ is the dual vector space of $V$. Then elements of $V$ (resp. $V^*$) are denoted by $|v\rangle$ (resp. $\langle v|$).

The set of highest weight vectors of a given type $(\lambda,\mu,\nu)$ in $\operatorname{Sym}^d(\otimes^3\mathbb{C}^m)\subseteq
\otimes^d(\otimes^3\mathbb{C}^m)$  forms a vector space which we denote by $\operatorname{HWV}_{\lambda,\mu,\nu}(\operatorname{Sym}^d(\otimes^3\mathbb{C}^m))$.
Let $\widehat{|\lambda,\mu,\nu\rangle}\in
\otimes^d(\otimes^3\mathbb{C}^m)$ denote the highest weight vector of type $(\lambda,\mu,\nu)$.  Let $S_d$ ($d\in \mathbb{N}$) be the symmetric group and $\mathcal{P}_d:=\sum_{\pi\in S_d}\pi$.
It was shown in Claim 22.1.5 of \cite{BlaI18} (see also Claim 4.2.17 of \cite{Iken12} for a dual version) that the tensors $\mathcal{P}_d(\pi,\sigma,\tau)\widehat{|\lambda,\mu,\nu\rangle}\in \otimes^d(\otimes^3\mathbb{C}^m)$  with $(\pi,\sigma,\tau)\in S_d^3$ generate $\operatorname{HWV}_{\lambda,\mu,\nu}(\operatorname{Sym}^d(\otimes^3\mathbb{C}^m))$.
For irreducible representations of
$\operatorname{GL}_m$, let $|\hat{\lambda}\rangle\in
\otimes^d\mathbb{C}^m$ denote the highest weight vector of type $\lambda$.  When doing evaluation,  the dual version of the highest weight vector is used (see \cite[Claim 4.2.17]{Iken12}).

Let $\widehat{\langle\lambda,\mu,\nu|}(\pi,\sigma,\tau)\in \otimes^d(\otimes^3\mathbb{C}^m)^*$, $\langle\hat{\lambda}|\in (\otimes^d\mathbb{C}^{m})^*$ and $w=\otimes_{i=1}^d(\otimes_{j=1}^3 w_{i}^{(j)})\in\otimes^d(\otimes^3\mathbb{C}^m)$. Then we have
\begin{align}\label{eq-lmnw}
\widehat{\langle\lambda,\mu,\nu|}(\pi,\sigma,\tau)|w\rangle=
\langle\hat{\lambda}|\pi|w^{(1)}\rangle \langle\hat{\mu}|\sigma|w^{(2)}\rangle
\langle\hat{\nu}|\tau|w^{(3)}\rangle,
\end{align}
where $w^{(j)}=\otimes_{i=1}^d w_{i}^{(j)}
\in\otimes^d \mathbb{C}^m$ ($1\leq j\leq 3$). In order to describe the evaluation in (\ref{eq-lmnw}) conveniently, Ikenmeyer introduced the concept of \emph{obstruction predesign}
\cite[Def. 7.2.5]{Iken12} and \emph{obstruction design}
(see \cite[Def. 7.2.9]{Iken12} and \cite[Sec.4.2]{BI13}). It is not hard to see that the definitions of obstruction design in \cite[Def. 7.2.9]{Iken12} and \cite[Sec.4.2]{BI13} are equivalent.

For obstruction design, we mainly use the definition in Section 4.2 of \cite{BI13}. The detail is given as follows. For a positive integer $m$, set $[m]:=\{1,...,m\}$. Given a subset $H$ of the
3-dimensional box
$$B(p,q,r):=\{(i,j,k)|i\in [p],j\in[q],k\in[r]\}\subseteq \mathbb{Z}_{+}^3,$$
which contains $pqr$ points. Here, we  identify the points of $B(p,q,r)$ with their coordinates $(i,j,k)$. We consider the slices that parallel to the coordinate planes
$$\textbf{e}_i^{(1)}=\{(x,y,z)\in H|x=i\},\quad 1\leq i\leq p,$$
$$\textbf{e}_j^{(2)}=\{(x,y,z)\in H|y=j\},\quad 1\leq j\leq q,$$
$$\textbf{e}_k^{(3)}=\{(x,y,z)\in H|z=k\},\quad 1\leq k\leq r.$$
The set
$$E^{(1)}=\{\textbf{e}_i^{(1)}|1\leq i\leq p\},$$
consisting of the $x$-slices of $H$ defines a set partition of $H$.
The $x$-marginal distribution of $H$ is: $m^{(1)}=(|\textbf{e}_1^{(1)}|,|\textbf{e}_2^{(1)}|,...,|\textbf{e}_p^{(1)}|)$. Similarly, we define the set partition $E^{(2)}$ of $y$-slices of $H$ with its marginal distribution $m^{(2)}$ and the set partition $E^{(3)}$ of $z$-slices of $H$ with its marginal distribution $m^{(3)}$. By a permutation of the sides, we may always assume that the marginal distributions $m^{(k)}$ are monotonically decreasing, i.e., partitions of
$d:= |H|$. Then the type of the set partition $E^{(k)}$ is defined by $m^{(k)}$. Let $\lambda={}^{t}m^{(1)}$, $\mu={}^{t}m^{(2)}$ and $\nu={}^{t}m^{(3)}$. Then we say that $H$ is an \emph{obstruction design} with type $(\lambda,\mu,\nu)$. We can see that an obstruction design with type $(\lambda,\mu,\nu)$ can be viewed as a 3-dimensional binary (0/1) contingency table with margins $({}^{t}\lambda,{}^{t}\mu,{}^{t}\nu)$ studied in \cite{PPV20}, or a 3-dimensional (0,1)-matrix with margins $({}^{t}\lambda,{}^{t}\mu,{}^{t}\nu)$.

In \cite{BI13},  B\"{u}rgisser and Ikenmeyer pointed out that: the evaluation of the obstruction design can be defined after fixing an ordering of $H\simeq [d]$. In this paper, we use the lexicographic ordering of $H$. More precisely, assume that $H=\{(x_i,y_i,z_i)|1\leq i\leq d\}$ and
$(x_1,y_1,z_1)<(x_2,y_2,z_2)<\cdots<(x_d,y_d,z_d)$ in the lexicographic ordering. Then identifying $(x_i,y_i,z_i)$ with $i$ for $1\leq i\leq d$, we have
\begin{align}\label{eq-hisod}
H\simeq [d].
\end{align}
Under $H\simeq [d]$,  a \emph{triple labeling} of $H$ is defined as a map  $w:~[d]\to (\mathbb{C}^m)^{\times3}$. Let $w_i:=w(i)=(w^{(1)}_i,w^{(2)}_i,w^{(3)}_i)^t$. Then $w$ can also be written as
\begin{align}\label{eq-wm}
w=(w_1,w_2,...,w_d)=\left(
  \begin{array}{cccc}
    w^{(1)}_1, & w^{(1)}_2, & \cdots ,& w^{(1)}_d \\
       w^{(2)}_1, & w^{(2)}_2, & \cdots ,& w^{(2)}_d \\
           w^{(3)}_1, & w^{(3)}_2, & \cdots ,& w^{(3)}_d
\end{array}
\right).
\end{align}

To describe evaluation of the obstruction design $H$ at the
triple labeling $w$, we introduce the notation `` $\operatorname{val}$''.
For a list of vectors $v_1,v_2,...,v_k\in \mathbb{C}^{m} $, $k\leq m$, let the evaluation $\operatorname{val}(v_1,v_2,...,v_k)$ denote the upper $k\times k$ minor of the $m\times k$ matrix $(v_1,v_2,...,v_k)$.
Thus $\operatorname{val}(v_1,v_2,...,v_m)=\det(v_1,v_2,...,v_m)$ when $v_i\in \mathbb{C}^{m}$.

For  $\textbf{e}_{i}^{(k)}\subseteq H\simeq [d]$, write $\textbf{e}_{i}^{(k)}=\{s_1, s_2,\ldots,s_{|\textbf{e}_{i}^{(k)}|} \}$ where $s_t<s_{t+1}$ under lexicographic ordering, for all
$t<|\textbf{e}_{i}^{(k)}|$. The evaluation of the obstruction design $H$ at the triple labeling $w$ is defined as
\begin{align}\label{eq-valh}
   \operatorname{val}_{H}(w)=&\prod_{\textbf{e}_{i}^{(1)}\in E^{(1)}} \operatorname{val} \left(w_{s_1}^{(1)},
w_{s_2}^{(1)},\cdots,w_{s_{|\textbf{e}_{i}^{(1)}|}}^{(1)}\right)\times\notag\\
   &\prod_{\textbf{e}_{j}^{(2)}\in E^{(2)}} \operatorname{val} \left(w_{s_1}^{(2)},
w_{s_2}^{(2)},\cdots,w_{s_{|\textbf{e}_{j}^{(2)}|}}^{(2)}\right)\times\\
&\prod_{\textbf{e}_{k}^{(3)}\in E^{(3)}} \operatorname{val} \left(w_{s_1}^{(3)},
w_{s_2}^{(3)},\cdots,w_{s_{|\textbf{e}_{k}^{(3)}|}}^{(3)}\right).\notag
\end{align}

To find the connection between  (\ref{eq-lmnw}) and (\ref{eq-valh}), we need the definition of ordered set partition triple and its corresponding permutation triple (see Section 7 of \cite{Iken12}). In this paper, we make $E^{(k)}$ into an ordered set partition as follows. Firstly, to give an ordering on $\textbf{e}_i^{(k)}$, we use the  natural ordering along the coordinate axis, that is, $\textbf{e}_i^{(k)}<\textbf{e}_j^{(k)}$ if $i<j$. Secondly, since $\textbf{e}_i^{(k)}\subseteq H\simeq [d]$, inside $\textbf{e}_i^{(k)}$ we use the natural ordering on $[d]$. In this way, $E^{(k)}$ becomes an ordered set partition of
$H$ for $1\leq k\leq 3$. Now we get an ordered set partition triple $(E^{(1)},E^{(2)},E^{(3)})$ of $H\simeq [d]$. By the discussion in Section 7 of \cite{Iken12}, there is a permutation triple corresponding to $(E^{(1)},E^{(2)},E^{(3)})$, which is denoted by
$(\pi_{H}^{(1)},\pi_{H}^{(2)},\pi_{H}^{(3)})$ where $\pi_{H}^{(1)},\pi_{H}^{(2)},\pi_{H}^{(3)}\in S_d$. In this way, it is not hard to see that $\pi_{H}^{(1)}=\epsilon$, the identity of $S_d$.

With assumptions above, by Proposition 7.2.4 of \cite{Iken12}, the relation between (\ref{eq-lmnw}) and (\ref{eq-valh}) is given as
\begin{align}\label{eq-lmnvalw}
 \operatorname{val}_{H}(w)=\widehat{\langle\lambda,\mu,\nu|}
 \left(\epsilon,\pi_{H}^{(2)},\pi_{H}^{(3)}\right)
  |w\rangle,
 \end{align}
where we identify $w$ in (\ref{eq-wm}) as $w=\otimes_{i=1}^d(\otimes_{j=1}^3 w_{i}^{(j)})\in\otimes^d(\otimes^3\mathbb{C}^m)$.


Suppose that the tensor $\omega\in \otimes^3\mathbb{C}^{m}$ is decomposed into distinct rank one tensors as $\omega=\sum_{i=1}^r \omega^{(1)}_i \otimes \omega^{(2)}_i\otimes \omega^{(3)}_i=\sum_{i=1}^r \omega_i$. Then we have
\begin{align}\label{eq-wotd}
\omega^{\otimes d}=\sum_{I:[d]\to [r]}\omega_I
\end{align}
where the sum is taken over all maps $I:[d]\to [r]$ and $$\omega_I=\omega_{I(1)}\otimes \omega_{I(2)}\cdots\otimes \omega_{I(d)}\in \otimes^d(\otimes^3\mathbb{C}^{m}),$$
where $\omega_{I(i)}=\omega_{I(i)}^{(1)}\otimes\omega_{I(i)}^{(2)}
\otimes\omega_{I(i)}^{(3)}$.
If we identify $\omega_I$ as
$$\omega_I=(\omega_{I(1)},\omega_{I(2)},\cdots,\omega_{I(d)})=\left(
  \begin{array}{cccc}
    \omega_{I(1)}^{(1)} & \omega_{I(2)}^{(1)} & \cdots & \omega_{I(d)}^{(1)} \\
    \omega_{I(1)}^{(2)} & \omega_{I(2)}^{(2)} & \cdots & \omega_{I(d)}^{(2)} \\
    \omega_{I(1)}^{(3)} & \omega_{I(2)}^{(3)} & \cdots & \omega_{I(d)}^{(3)} \\
  \end{array}
\right),$$
then it can be considered as a triple labeling of $H\simeq [d]$.

Let $\mathcal{O}(\otimes^3\mathbb{C}^{m})$ be the ring of polynomials on $\otimes^3\mathbb{C}^{m}$. For $d\in \mathbb{Z}$, the degree $d$ part of $\mathcal{O}\left(\otimes^3\mathbb{C}^{m}\right)$ is denoted by
$\mathcal{O}\left(\otimes^3\mathbb{C}^{m}\right)_d$. By (4.2.5) of \cite{Iken12}, we know that
\begin{align}\label{eq-symdod}
\operatorname{Sym}^d(\otimes^3\mathbb{C}^m)^*\simeq
\mathcal{O}\left(\otimes^3\mathbb{C}^{m}\right)_d
\end{align}
as $G$-representations, where $G=\operatorname{GL}_m^3$ or its subgroup.
Given an obstruction design $H$ of type $(\lambda,\mu,\nu)$, then
$\widehat{\langle\lambda,\mu,\nu|}
(\epsilon,\pi_{H}^{(2)},\pi_{H}^{(3)})
\mathcal{P}_d\in \operatorname{Sym}^d(\otimes^3\mathbb{C}^m)^*$. By (\ref{eq-symdod}), $\widehat{\langle\lambda,\mu,\nu|}
(\epsilon,\pi_{H}^{(2)},\pi_{H}^{(3)})
\mathcal{P}_d$ can be considered as a homogeneous polynomial in $\mathcal{O}\left(\otimes^3\mathbb{C}^{m}\right)_d$, which is denoted by $F_{H}$.
So by Lemma 4.2.2 of \cite{Iken12}, the evaluation $F_{H}(\omega)$ is given by
\begin{align}\label{eq-fhw}
F_{H}(\omega)&=\operatorname{val}_{H}(\omega^{\otimes d})=
\widehat{\langle\lambda,\mu,\nu|}(\epsilon,\pi_{H}^{(2)},\pi_{H}^{(3)})
\mathcal{P}_d
|\omega^{\otimes d}\rangle,\\
&=\sum_{I:[d]\to [r]}\widehat{\langle\lambda,\mu,\nu|}
(\epsilon,\pi_{H}^{(2)},\pi_{H}^{(3)})|\omega_I\rangle\notag\\
&=\sum_{I:[d]\to [r]}\operatorname{val}_{H}(\omega_I).\notag
\end{align}

Suppose that $G=\operatorname{GL}_m^3$ or its subgroup, such as $\operatorname{SL}_m^3$. The action of $G$ on $\otimes^3\mathbb{C}^m$ induces an action of $G$ on $\mathcal{O}\left(\otimes^3\mathbb{C}^{m}\right)$ defined by $(g\cdot f)(x):=f(g^{-1}x)$ (or $f(g^{t}x)$ ) for $g\in G$, $f\in \mathcal{O}\left(\otimes^3\mathbb{C}^{m}\right)$ and $x\in \otimes^3\mathbb{C}^{m}$.

\begin{definition}\cite{BI17}
$f\in\mathcal{O}(\otimes^3\mathbb{C}^{m})$ is called the generic fundamental invariant of $\otimes^3\mathbb{C}^{m}$ if $s\cdot f=f$ for all $s\in \operatorname{SL}_m^3$. The set of all generic fundamental invariants is denoted by $\mathcal{O}\left(\otimes^3\mathbb{C}^{m}\right)^{\operatorname{SL}_m^3}$. Its degree $d$ part is denoted by
$\mathcal{O}\left(\otimes^3\mathbb{C}^{m}\right)^{\operatorname{SL}_m^3}_d$.
\end{definition}
Let $d_0$ be the minimal degree such that $\mathcal{O}\left(\otimes^3\mathbb{C}^{m}\right)^{\operatorname{SL}_m^3}_{d_0}\neq\emptyset$.
Then any $f\in\mathcal{O}\left(\otimes^3\mathbb{C}^{m}\right)^{\operatorname{SL}_m^3}_{d_0}
$ is called a \emph{minimal generic fundamental invariant} of $\otimes^3\mathbb{C}^{m}$.
The \emph{generic degree monoid} (see \cite[(5.1)]{BI17}) of $\otimes^3\mathbb{C}^{m}$ is defined as
$$\mathcal{E}(m)=\{d\in \mathbb{N}~|~\mathcal{O}(\otimes^3\mathbb{C}^{m})^{\operatorname{SL}_m^3}_d\neq\emptyset\}.$$
By the discussion of Section 5 in \cite{BI17}, we know that $m|d$ for $d\in\mathcal{E}(m)$. Moreover, setting $d=m\delta$, we have
$$\dim\mathcal{O}\left(\otimes^3\mathbb{C}^{m}\right)^{\operatorname{SL}_m^3}_{m\delta}
=k_{m}(\delta),$$
where $k_m(\delta):=k(m\times\delta, m\times\delta,m\times\delta)$ is the Kronecker coefficient assigned to three partitions of the same rectangular shape $m\times\delta:=(\delta,...,\delta)$ ($m$ times).
For $m>2$, define
$$\mathcal{E}'(m)=\{\delta\in \mathbb{N}~|~k_{m}(\delta)>0\}$$
and
\begin{align}\label{eq-delt}
\delta(m)=\min\mathcal{E}'(m).
\end{align}

Generally, suppose that $\lambda$, $\mu$ and $\nu$ are partitions of $d$ with at most $m_1$, $m_2$ and $m_3$ rows, respectively. Let $\operatorname{HWV}_{\lambda,\mu,\nu}(\operatorname{Sym}^d(\mathbb{C}^{m_1}\otimes\mathbb{C}^{m_2}
\otimes\mathbb{C}^{m_3}))$ denote the set of highest weight vectors
with type $(\lambda,\mu,\nu)$. Let $\mathcal{O}(\mathbb{C}^{m_1}\otimes\mathbb{C}^{m_2}
\otimes\mathbb{C}^{m_3})$ be the ring of polynomials on $\mathbb{C}^{m_1}\otimes\mathbb{C}^{m_2}
\otimes\mathbb{C}^{m_3}$.
Then the discussions in this section can be generalized to $\operatorname{HWV}_{\lambda,\mu,\nu}(\operatorname{Sym}^d(\mathbb{C}^{m_1}\otimes\mathbb{C}^{m_2}
\otimes\mathbb{C}^{m_3}))$ and $\mathcal{O}(\mathbb{C}^{m_1}\otimes\mathbb{C}^{m_2}
\otimes\mathbb{C}^{m_3})$, straightforwardly. We can see it in
Section \ref{sec-lmnn2}.

\section{The minimal generic fundamental invariant of $\otimes^3\mathbb{C}^{n^2-1}$}\label{se:od3}
In Section 5 of \cite{BI17}, the authors constructed the minimal generic fundamental invariant of $\otimes^3\mathbb{C}^{n^2}$. In this section, we continue their discussion and provide the minimal generic fundamental invariant of $\otimes^3\mathbb{C}^{n^2-1}$.

For $\delta(m)$ defined in (\ref{eq-delt}), we have the following theorem. It generalizes and improves Theorem 5.9 of \cite{BI17}. Moreover, it partially answers the Problem 5.17 of \cite{BI17}.
\begin{theorem}\label{thm-deln}
\

\begin{enumerate}
  \item\label{it-lu} If $m>2$, we have $\lceil \sqrt{m} \rceil \leq \delta(m)\leq m$.
  \item\label{it-n2-1}
$k_{n^2-j}(n)=k_j(n)$ for $0\leq j\leq n^2$. So if $\lceil \sqrt{n^2-j}\rceil =n$ and $k_j(n)>0$, then we have $\delta(n^2-j)=n$.
In particular, we have
\begin{enumerate}
  \item\label{it-it1} $k_{n^2-1}(n)=k_{1}(n)=1$ and $\delta(n^2-1)=n$;
  \item\label{it-it2} $k_{n^2-2}(n)=k_{2}(n)=1$ and $\delta(n^2-2)=n$, when $n$ is even;
  \item $k_{n^2-3}(n)=k_{3}(n)>0$ and $\delta(n^2-3)=n$ for $n\geq3$;
  \item $k_{n^2-n}(n)=k_{n}(n)>0$ and $\delta(n^2-n)=n$.
\end{enumerate}
\end{enumerate}
\end{theorem}

\begin{proof}
(1) The condition $\lceil \sqrt{m} \rceil \leq \delta(m)$ was proved in (1) of \cite[Thm. 5.9]{BI17}. Since $m\times m$ is self-conjugate, by \cite[Cor. 3.2]{Bessen04} or \cite[Thm. 4.6]{PPV16} we have
$k(m\times m,m\times m,m\times m)>0$. So we have $\delta(m)\leq m$.

(2) The following symmetry relation from \cite[Cor.4.4.15]{Iken12}:
\begin{align}\label{eq-lmnc}
k(\lambda,\mu,\nu) = k(n^2\times n-\lambda, n^2\times n-\mu,n^2\times n-\nu),
\end{align}
holds for any partitions $\lambda,\mu,\nu$ contained in $n^2\times n$. Here, $n^2\times n-\lambda$ denotes the partition corresponding to the complement of the Young diagram of $\lambda$ in the rectangle $n^2\times n$.

In (\ref{eq-lmnc}) above,  let $\lambda=\mu=\nu=j\times n$ with $0\leq j\leq n^2$. Then we have
\begin{align*}
k(j\times n,j\times n,j\times n)=&k(n^2\times n-(j\times n), n^2\times n-(j\times n), n^2\times n-(j\times n))\\
=&k((n^2-j)\times n, (n^2-j)\times n, (n^2-j)\times n).
\end{align*}

In particular, setting $j=1$, we get $k_{n^2-1}(n)=k(1\times n, 1\times n, 1\times n)=1$. Setting $j=2$, we get $k_{n^2-2}(n)=k(2\times n, 2\times n, 2\times n)$.
By \cite[Thm. 1.6]{Tewari}, we have $k(2\times n, 2\times n, 2\times n)=1$ (resp. 0) if $n$ is even (resp. odd). It is well known that $k_3(n)>0$ for $n\geq2$, see for example
\cite[Example 5.5]{BI17}. Setting $j=n$, we get $k_{n^2-n}(n)=k_{n}(n)>0$ by the proof in (\ref{it-lu}).
\end{proof}
\begin{remark}
It has been verified in Problem 5.19 of \cite{BI17} that $k_m(\delta)>0$ where $\delta\geq\delta(m)$ and $m\leq12$.
\end{remark}
So if $m=n^2-1$, then (\ref{it-it1}) of Theorem \ref{thm-deln} states that, up to a scaling factor, there is exactly one homogeneous $\operatorname{SL}^3_m$-invariant $F_m: \otimes^3 \mathbb{C}^m\to \mathbb{C}$ of degree $n^3-n$ (and no nonzero invariant of smaller degree). In order to construct $F_m$, we introduce some definitions.

\begin{definition}\label{def-diag}
For the 3-dimensional box $B(n,n,n)=\{(i,j,k)|~i,j,k\in [n]\}$$\subseteq \mathbb{Z}_{+}^3$, a subset $D\subseteq B(n,n,n)$  is called a diagonal if $D=\{(i,\sigma(i),\tau(i))|i\in [n]\}$ for some $\sigma,\tau\in S_n$. When $\sigma=\tau=\epsilon$ are the identity of $S_n$, the diagonal $\{(i,i,i)|i\in [n]\}$ is called the main diagonal of $B(n,n,n)$, denoted by $D(n)$.
\end{definition}

For a diagonal $D=\{(1,\sigma(1),\tau(1)),(2,\sigma(2),\tau(2)),...,
(n,\sigma(n),\tau(n))\}\subseteq B(n,n,n)$, if we permute two $y-$slices, then $D$ becomes another diagonal $D'$. Moreover, there exists a transposition $\alpha\in S_n$ such that $D'=\{(1,\alpha\sigma(1),\tau(1)),(2,\alpha\sigma(2),\tau(2)),...,
(n,\alpha\sigma(n),\tau(n))\}$. Similarly, if we permute two $z-$slices, there exists a transposition $\alpha'\in S_n$ such that $D'=\{(1,\sigma(1),\alpha'\tau(1)),(2,\sigma(2),\alpha'\tau(2)),...,
(n,\sigma(n),\alpha'\tau(n))\}$. So we have the following lemma.
\begin{lemma}\label{le-d1d2}
For any two diagonals $D_1$, $D_2$ of $B(n,n,n)$, by permuting slices one diagonal can be changed into another.
\end{lemma}
Moreover, in Lemma \ref{le-d1d2}, by definition we can just permute the $y$ and $z-$slices and leave the $x-$slices unchanged.

\begin{definition}\label{def-equivalent}
Two obstruction designs $H_1$ and $H_2$ of the same type are said to be equivalent if we can get $H_2$ from $H_1$ by permuting slices.
\end{definition}

For any two diagonals $D_1$, $D_2$  of $B(n,n,n)$,
let $H_{D_1}=B(n,n,n)\setminus D_1$ and $H_{D_2}=B(n,n,n)\setminus D_2$. Then by Lemma \ref{le-d1d2} we can turn $H_{D_1}$ into $H_{D_2}$ by permuting slices and vice versa. Following similar proof of Claim 7.2.17 of \cite{Iken12}, we have the following lemma.
\begin{lemma}\label{le-unique}
Let $H$ be an obstruction design with the type $(\lambda,\mu,\nu)$, where $\lambda=\mu=\nu=(n^2-1)\times n$.
Then there exists a diagonal $D\subseteq B(n,n,n)$ such that
$H=B(n,n,n)\setminus D$. That is, up to equivalence, there is exactly one obstruction designs of type ($(n^2-1)\times n$, $(n^2-1)\times n$, $(n^2-1)\times n$).
\end{lemma}

Suppose that  $H_1$ and $ H_2$ are equivalent and $|H_1|=|H_2|=d$.  For any triple labeling $w=(w_1,w_2,...,w_{d})$ on $H_1$, along with the permutation of slices, $w$ becomes a triple labeling of $H_2$, and we denote it by $w'$. Moreover, we can see that $w'$ is obtained by permuting $w_i$ ($i=1,2,...,d$), that is, there exists a $\pi\in S_d$ such that
$w'=\pi w=(w_{\pi^{-1}(1)},w_{\pi^{-1}(2)},...,w_{\pi^{-1}(d)}$). In general, there are many ways to get $H_2$ from $H_1$ by permuting slices. If we fix one way, then the permutation $\pi$ is also fixed. Thus, along with this fixed way, $w$ becomes $w'=\pi w$ for each triple labeling $w$ of $H_1$. By the definition of $\operatorname{val}_H$ in (\ref{eq-valh}), we have
\begin{align}\label{eq-valw}
 \operatorname{val}_{H_{1}}(w)=\pm \operatorname{val}_{H_2}(w')=\pm \operatorname{val}_{H_2}(\pi w).
\end{align}
In particular, we have the following lemma.
\begin{lemma}\label{le-piw}
For two diagonals $D_1$, $D_2$, let $H_{D_1}=B(n,n,n)\setminus D_1$ and $H_{D_2}=B(n,n,n)\setminus D_2$. Suppose that $w$ is a triple labeling on $H_{D_1}$. Then by permuting slices $H_{D_1}$ becomes $H_{D_2}$ and so $w$ becomes a labeling $w'$ of $H_{D_2}$. Then $\operatorname{val}_{H_{D_1}}(w)= \operatorname{val}_{H_{D_2}}(w')$.
\end{lemma}
\begin{proof}
By Lemma \ref{le-d1d2}, $H_{D_2}$ can be obtained from  $H_{D_1}$ by permuting slices.  So it suffices to show the case when $H_{D_2}$ is obtained from $H_{D_1}$ by permuting two slices, such as the $x$-slices. Without loss of generality, suppose that  $\textbf{e}_1^{(1)}$ and $\textbf{e}_2^{(1)}\in H_{D_1}$ are exchanged and along with the exchange, $w$ becomes $w'$. By (\ref{eq-valw}), we have that $\operatorname{val}_{H_{D_1}}(w)=\pm \operatorname{val}_{H_{D_2}}(w')$. In the following, by analysing the change of signs, we can see that in fact they are equal.

By (\ref{eq-valh}), in the expressions of $\operatorname{val}_{H_{D_1}}(w)$ and $\operatorname{val}_{H_{D_2}}(w')$,  permuting $\textbf{e}_1^{(1)}$ and $\textbf{e}_2^{(1)}$ just changes the order of the products
of the evaluations of $\textbf{e}_1^{(1)}$ and $\textbf{e}_2^{(1)}$. So the evaluations corresponding to $x$-slices are unchanged.
On the other hand, permuting two $x$-slices exchanges two columns of points inside each $y$-slice and each $z$-slice. It is well-known that permuting two column vectors inside the determinant contributes one sign `-1' to the determinant. In the following, we can see that the total change of signs is even, and therefore $\operatorname{val}_{H_{D_2}}(w')=\operatorname{val}_{H_{D_1}}(w)$.

Since $H_{D_1}=B(n,n,n)\setminus D_1$, there is exactly one deleted point in each slice. In $\textbf{e}_1^{(1)}$ and $\textbf{e}_2^{(1)}$, suppose that the deleted points lie in the position $(1,j_1,k_1)$ and $(2,j_2,k_2)$, respectively. Then we have $j_1\neq j_2$ and $k_1\neq k_2$. Now we can find $n-2$ points in $\textbf{e}_1^{(1)}$, whose positions are $\{(1,j_i,k_i)|3\leq i\leq n\}$ such that $(j_1,j_2,...,j_n)$ and $(k_1,k_2,...,k_n)$ are permutations of $[n]$.

By (\ref{eq-valh}), in the expressions of $\operatorname{val}_{H_{D_1}}(w)$ and $\operatorname{val}_{H_{D_2}}(w')$, it is not hard to see the following:
\begin{itemize}
  \item By exchanging $n(n-1)$ positions, the evaluation of vectors labeled by $\textbf{e}_{j_1}^{(2)}\in H_{D_2}$ can be obtained from the evaluation of vectors labeled by $\textbf{e}_{j_1}^{(2)}\in H_{D_1}$.  Similarly, by exchanging $n(n-1)$ positions, the evaluation of vectors labeled by $\textbf{e}_{k_1}^{(3)}\in H_{D_2}$ can be obtained from the evaluation of vectors labeled by $\textbf{e}_{k_1}^{(3)}\in H_{D_1}$;
  \item Similarly, by exchanging $n(n-1)$ positions, the evaluation of vectors labeled by $\textbf{e}_{j_2}^{(2)}\in H_{D_2}$ (resp. $\textbf{e}_{k_2}^{(3)}\in H_{D_2}$) can be obtained from the evaluation of vectors labeled by $\textbf{e}_{j_2}^{(2)}\in H_{D_1}$ (resp. $\textbf{e}_{k_2}^{(3)}\in H_{D_1}$);
  \item For $3\leq i \leq n$, by exchanging $n^2$ positions, the evaluation of vectors labeled by $\textbf{e}_{j_i}^{(2)}\in H_{D_2}$ (resp. $\textbf{e}_{k_i}^{(3)}\in H_{D_2}$) can be obtained from the evaluation of vectors labeled by $\textbf{e}_{j_i}^{(2)}\in H_{D_1}$ (resp. $\textbf{e}_{k_i}^{(3)}\in H_{D_1}$).
\end{itemize}
Because all contributions of signs are made by three items above, the total change of signs is given by $$(-1)^{2n(n-1)+2n(n-1)+2(n-2)n^2}=1,$$
which implies $\operatorname{val}_{H_{D_1}}(w)= \operatorname{val}_{H_{D_2}}(w')$.
\end{proof}
Moreover, we have the following lemma which will be used later.
Note that permuting slices leaves $H=B(\ell,m,n)$ unchanged.
\begin{lemma}\label{le-signun}
Let $H=B(\ell,m,n)$. Suppose that $w$ is a triple labeling on $H$. By permuting slices $w$ becomes another triple labeling $w'$ of $H$. Then $\operatorname{val}_{H}(w)= \operatorname{val}_{H}(w')$.
\end{lemma}
\begin{proof}
Just as the proof of Lemma \ref{le-piw}, it suffices to show the case when $w'$ is obtain from $w$ by permuting two $x$-slices.
Without loss of generality, suppose that  $\textbf{e}_1^{(1)}$ and $\textbf{e}_2^{(1)}\in H$ are exchanged.

As the discussion in Lemma \ref{le-piw}, when exchanging $x$-slices the contribution of signs is made by $y$ and $z$-slices. More precisely, we have
\begin{itemize}
  \item For each $y$-slice, two columns of length $n$ are exchanged. So by exchanging $n^2$ positions, the evaluation of vectors labeled by $\textbf{e}_{j}^{(2)}\in H$ with respect to $w'$ can be obtained from the evaluation of vectors labeled by $\textbf{e}_{j}^{(2)}$ with respect to $w$. Since there are $m$ $y$-slices, for $y$-slices the total contribution of signs is $(-1)^{mn^2}$.
  \item For each $z$-slice, two rows of length $m$ are exchanged. So by exchanging $m^2$ positions, the evaluation of vectors labeled by $\textbf{e}_{k}^{(3)}\in H$ with respect to $w'$ can be obtained from the evaluation of vectors labeled by $\textbf{e}_{k}^{(3)}$ with respect to $w$. Since there are $n$ $z$-slices, for $z$-slices the total contribution of signs is $(-1)^{nm^2}$.
\end{itemize}
So the total contribution of signs is given by
$$(-1)^{mn^2+nm^2}=(-1)^{mn(m+n)}=1,$$
which implies $\operatorname{val}_{H}(w)= \operatorname{val}_{H}(w')$.
\end{proof}

Now, for $m=n^2-1$ we construct $F_m$ (or $F_{m,D}$ see below)  as follows. Select one diagonal $D$ of $B(n,n,n)$ and let $H_D=B(n,n,n)\setminus D$.
Then the obstruction design $H_D$ has the type $(\lambda,\mu,\nu)$, where $\lambda=\mu=\nu=m\times n\vdash n^3-n$.
Let ($\epsilon$, $\pi_{H_D}^{(2)}$, $\pi_{H_D}^{(3)}$) be the corresponding permutation triple of $H_D$. Following the discussion in Section \ref{se:prelim}, we define
$F_{m,D}\in \operatorname{Sym}^{n^3-n}(\otimes^3\mathbb{C}^m)^*\simeq O(\otimes^3\mathbb{C}^m)_{n^3-n}$ by
\begin{align}\label{eq-fm}
F_{m,D}=\widehat{\langle \lambda,\mu,\nu |}\left(\epsilon,\pi_{H_D}^{(2)},\pi_{H_D}^{(3)}\right)
\mathcal{P}_{n^3-n}.
\end{align}

\begin{lemma}\label{lem-fd1d2}
For any two diagonals $D_1$, $D_2$ of $B(n,n,n)$, we have
$F_{m,D_1}= F_{m,D_2}$.
\end{lemma}
\begin{proof}
Let $H_{D_1}=B(n,n,n)\setminus D_1$ and $H_{D_2}=B(n,n,n)\setminus D_2$ be the obstruction designs.

For any tensor $t=\sum_{i=1}^r t_i$ of $\otimes^3 \mathbb{C}^m$, following (\ref{eq-fhw}) and (\ref{eq-fm}) we have
\begin{align}
F_{m,D_1}(t)=\widehat{\langle \lambda,\mu,\nu |}\left(\epsilon,\pi_{H_{D_1}}^{(2)},\pi_{H_{D_1}}^{(3)}\right)
\mathcal{P}_{n^3-n}|t^{\otimes n^3-n}\rangle=\sum_{I:[n^3-n]\to [r]} \operatorname{val}_{H_{D_1}}(t_I).
\end{align}
Since $t_I$ can be considered as the triple labeling on $H_{D_1}$, by Lemma \ref{le-piw} there exists a $\pi\in S_{n^3-n}$ such that
$$\operatorname{val}_{H_{D_1}}(t_I)= \operatorname{val}_{H_{D_2}}(\pi t_I).$$
Then the proof follows from $\mathcal{P}_{n^3-n}\pi=\mathcal{P}_{n^3-n}$.
\end{proof}

Just like Theorem 5.13 of \cite{BI17} and Example 4.12 of \cite{BI13}, we have the following theorem.
\begin{theorem}\label{thm-finvn2-1}
Let $m=n^2-1$. For the $F_{m,D}$ defined in (\ref{eq-fm}), we have $F_{m,D} ((g_1, g_2, g_3)w)=(det(g_1)det(g_2) det(g_3))^n F_{m,D}(w)$ for all $(g_1, g_2, g_3)\in \operatorname{GL}_{m}^3$ and $w\in \otimes^3 \mathbb{C}^m$. Thus, $F_{m,D}$ is invariant under the action of $\operatorname{SL}_m^3$. Moreover, $F_{m,D}\neq0.$
\end{theorem}
\begin{proof}
Assume that $w=\sum_{i=1}^r w_i$, where $w_i=w_{i}^{(1)}\otimes w_{i}^{(2)}\otimes w_{i}^{(3)} $ and $w_{i}^{(k)}\in \mathbb{C}^m$. Let
$$u=(g_1\otimes g_2\otimes g_3)w=\sum_{i=1}^r \left(g_1w_{i}^{(1)}\otimes g_2 w_{i}^{(2)}\otimes g_3 w_{i}^{(3)}\right)=\sum_{i=1}^r u_i.$$
Then by (\ref{eq-fhw}) we have
\begin{align*}
  F_{m, D} (u)=\sum_{I:[n^3-n]\to [r]} \operatorname{val}_{H_D}(u_I),
\end{align*}
where $u_I=u_{I(1)}\otimes u_{I(2)}\otimes\cdots\otimes u_{I(n^3-n)}$ and  $u_{I(j)}=g_1w_{I(j)}^{(1)}\otimes g_2w_{I(j)}^{(2)}\otimes g_3w_{I(j)}^{(3)}$.

By the definition of $\operatorname{val}$, we have that $\operatorname{val}(v_1,v_2,...,v_m)=\det (v_1,v_2,...,v_m)$ for $v_i\in \mathbb{C}^m$. It is well known that $\det(AB)=\det(A)\det(B)$ for $m\times m$ matrices $A$ and $B$.
Since $H_D=B(n,n,n)\setminus D$, for each slice $\textbf{e}_i^{(k)}$ of $H_D$, we have $|\textbf{e}_i^{(k)}|=n^2-1=m$. By the definition of $\operatorname{val}_H$ in (\ref{eq-valh}), we have $$\operatorname{val}_{H_D}(u_I)=\left(\det(g_1)\det(g_2)\det(g_3)\right)^n \operatorname{val}_{H_D}(w_I),$$
where $w_I=w_{I(1)}\otimes w_{I(2)}\otimes\cdots\otimes w_{I(n^3-n)}$ and  $w_{I(j)}=w_{I(j)}^{(1)}\otimes w_{I(j)}^{(2)}\otimes w_{I(j)}^{(3)}$. So we have
\begin{align*}
  F_{m, D} (u)&=\sum_{I:[n^3-n]\to [r]} \operatorname{val}_{H_D}(u_I)=\left(\det(g_1)\det(g_2)\det(g_3)\right)^n\sum_{I:[n^3-n]\to [r]} \operatorname{val}_{H_D}(w_I)\\
  &=\left(\det(g_1)\det(g_2)\det(g_3)\right)^n F_{m, D} (w).
\end{align*}

By Theorem \ref{thm-deln}, we have $k((n^2-1)\times n, (n^2-1)\times n, (n^2-1)\times n)=1$. Thus, by Theorem 4.1 of \cite{BI13}, up to a scaling factor, there exists a highest weight vector $f_H\neq0$ in $\operatorname{Sym}^{n^3-n}(\otimes^3\mathbb{C}^m)^*$, which corresponds to  an obstruction design $H$ of type $((n^2-1)\times n, (n^2-1)\times n, (n^2-1)\times n)$. By Lemma \ref{le-unique}, there exists some diagonal $D\in B(n,n,n)$ such that $H=B(n,n,n)\setminus D=H_D$ and $f_H=F_{m,D}\neq0$. By Lemma \ref{lem-fd1d2}, this implies that $F_{m,D}\neq0$ for all diagonals $D$.
\end{proof}
Thus, the proof of $F_{m,D}\neq0$ is also indirect. Since $F_{m,D}$ does not depend on the choice of diagonal $D$,
we can  denote $F_{m,D}$ by $F_{m}$.

\section{The minimal generic fundamental invariant of $\mathbb{C}^{\ell m}\otimes \mathbb{C}^{mn}\otimes \mathbb{C}^{n\ell}$}\label{sec-lmnn2}

Let $\lambda$, $\mu$, $\nu$ be three partitions of $d$.
Let $k(\lambda,\mu,\nu)$ denote the corresponding Kronecker coefficient. Let $\{\lambda\}$, $\{\mu\}$ and $\{\nu\}$
denote the corresponding irreducible representations for general linear groups.
Then  Proposition 4.4.8 of \cite{Iken12} tells us that
$$ \operatorname{Sym}^d\left(\mathbb{C}^{m_1}\otimes \mathbb{C}^{m_2}\otimes \mathbb{C}^{m_3}\right)\simeq
\bigoplus_{\lambda,\mu,\nu\vdash d}k(\lambda,\mu,\nu)\{\lambda\}\otimes\{\mu\}\otimes\{\nu\},$$
where $\ell(\lambda)\leq m_1$, $\ell(\mu)\leq m_2$, $\ell(\nu)\leq m_3$. So by the isomorphism
$$\operatorname{Sym}^d\left(\mathbb{C}^{m_1}\otimes \mathbb{C}^{m_2}\otimes \mathbb{C}^{m_3}\right)^*\simeq \mathcal{O}\left(\mathbb{C}^{m_1}\otimes \mathbb{C}^{m_2}\otimes \mathbb{C}^{m_3}\right)_d$$ we have
$$ \mathcal{O}\left(\mathbb{C}^{m_1}\otimes \mathbb{C}^{m_2}\otimes \mathbb{C}^{m_3}\right)_d^{\operatorname{SL}_{m_1}\times \operatorname{SL}_{m_2} \times \operatorname{SL}_{m_3}}\simeq
\bigoplus_{\lambda,\mu,\nu\vdash d}k(\lambda,\mu,\nu)(\{\lambda^*\}\otimes\{\mu^*\}\otimes\{\nu^*\})^
{\operatorname{SL}_{m_1}\times \operatorname{SL}_{m_2} \times \operatorname{SL}_{m_3}}
.$$
By discussions in Section 1.1 of \cite{Bald18} (see also \cite[Sect. 8.3]{BLMW11} and \cite{Maniv10}), we know that
if $ \mathcal{O}\left(\mathbb{C}^{m_1}\otimes \mathbb{C}^{m_2}\otimes \mathbb{C}^{m_3}\right)_d^{\operatorname{SL}_{m_1}\times \operatorname{SL}_{m_2} \times \operatorname{SL}_{m_3}}\neq0$, there exist $\delta_1$, $\delta_2$ and  $\delta_3$
such that
\begin{align}\label{eq-d1d2d3}
\lambda=m_1\times\delta_1,\quad\mu=m_2\times\delta_2,\quad
\nu=m_3\times\delta_3\quad\text{and}
\quad k(\lambda,\mu,\nu)>0,
\end{align}
where $\lambda$, $\mu$, $\nu$ are partitions of $d$.
Moreover, by Lemma 4.4.7 and Corollary 4.4.12 of \cite{Iken12}, in (\ref{eq-d1d2d3}) if $k(\lambda,\mu,\nu)>0$ we should have
\begin{align}\label{eq-m1m2m3}
m_1\leq\delta_2\delta_3,\quad m_2\leq\delta_1\delta_3,\quad m_3\leq\delta_1\delta_2.
\end{align}

In particular, if we let $m_1=\ell m$, $m_2= m n$ and $m_3=n\ell$,
by Corollary 4.4.15 of \cite{Iken12}, we have
$k(\ell m\times n, mn\times\ell,n\ell\times m)=1$.
So it was pointed in Remark 5.14 of \cite{BI17} that the generic fundamental invariant of $\mathbb{C}^{n^2}\otimes \mathbb{C}^{n^2}\otimes \mathbb{C}^{n^2}$ has a natural generalization to $\mathbb{C}^{\ell m}\otimes \mathbb{C}^{mn}\otimes \mathbb{C}^{n\ell}$. That is, there exists an invariant $F_{(\ell,m,n)}: \mathbb{C}^{\ell m}\otimes \mathbb{C}^{mn}\otimes \mathbb{C}^{n\ell}\to \mathbb{C}$ of degree $\ell mn$.
\begin{proposition}
The degree $\ell mn$ of the invariant $F_{(\ell,m,n)}$ is minimal among all generic invariants of $\mathbb{C}^{\ell m}\otimes \mathbb{C}^{mn}\otimes \mathbb{C}^{n\ell}$.
\end{proposition}
\begin{proof}
Determining the minimal degree is equivalent to find the minimal $d$ such that $d=\ell m \delta_1=m n \delta_2= n\ell \delta_3$ and
$$k(\ell m\times \delta_1,~m n\times \delta_2,~n\ell \times \delta_3)>0.$$
Let $m_1=\ell m$, $m_2= m n$ and $m_3=n\ell$ in (\ref{eq-m1m2m3}). Then we have $\ell mn\leq\delta_1\delta_2\delta_3$.
Let $\delta_1\delta_2\delta_3=\ell mn$. Then $\delta_1=n$, $\delta_2=\ell$ and $\delta_3=m$ is the unique solution satisfying
$\ell m\delta_1=m n\delta_2=n\ell\delta_3$. Since $k(\ell m\times n,~m n\times \ell,~n\ell \times m)=1>0$, we have $d=\ell mn$ is the minimal degree.
\end{proof}

So when $\ell=m=n$, we have  $F_{(\ell,m,n)}=F_n$, which was studied in \cite{BI17}.
Let $\langle\ell,m,n\rangle\in\mathbb{C}^{\ell m}\otimes \mathbb{C}^{mn}\otimes \mathbb{C}^{n\ell}$ and $\langle n^2 \rangle$=$\sum_{i=1}^{n^2} e_i\otimes e_i \otimes e_i\in \mathbb{C}^{n^2}\otimes \mathbb{C}^{n^2}\otimes \mathbb{C}^{n^2}$  be the matrix multiplication tensor and unit tensor (see for example \cite{BCS97,Lan17}), respectively.  In this section,
motivated by Problem \ref{prob-fnn0}, we give combinatorial descriptions for the evaluations of $F_{(\ell,m,n)}(\langle \ell,m,n \rangle)$ and $F_{n}(\langle n^2 \rangle)$. Our descriptions are available for checking by computer when $\ell$, $m$, $n$ are small. For example, by computer we verified that $F_{(2,2,2)}(\langle 2,2,2 \rangle)$ and $F_{(2,2,3)}(\langle 2,2,3 \rangle)\neq0$.
On the other hand, we give a more clearer description of the relation between $F_{n}(\langle n^2 \rangle)$ and Latin cube discussed in Section 5 of \cite{BI17}.

\subsection{The construction of $F_{(\ell,m,n)}$ and its evaluation}
\

Let $\lambda=\ell m\times n$, $\mu=mn\times\ell$ and $\nu=n\ell\times m$. Let $H=B(n,\ell,m)$ be a cuboid. Then $H=B(n,\ell,m)$ can be viewed as an obstruction design whose $x$-slices, $y$-slices and $z$-slices consist of $\ell m$, $mn$ and $n\ell$ points, respectively. Just as the discussion in Section \ref{se:prelim}, $H=B(n,\ell,m)$ determines a permutation triple
$(\epsilon,\pi_{H}^{(2)},\pi_{H}^{(3)})$ with respect to the lexicographic ordering. Similarly, equation (\ref{eq-lmnvalw}) can be generalized as follows:
\begin{align*}
\operatorname{val}_{H}(w)=\widehat{\langle\lambda,\mu,\nu|}
 \left(\epsilon,\pi_{H}^{(2)},\pi_{H}^{(3)}\right)
  |w\rangle,
\end{align*}
where  $w:H=B(n,\ell,m)\simeq[\ell mn]\to \mathbb{C}^{\ell m}\times\mathbb{C}^{mn}\times\mathbb{C}^{n\ell}$ is a triple labeling.

Since $F_{(\ell,m,n)}\in \operatorname{Sym}^{\ell mn}(\mathbb{C}^{\ell m}\otimes \mathbb{C}^{mn}\otimes \mathbb{C}^{n\ell})^*\simeq \mathcal{O}(\mathbb{C}^{\ell m}\otimes \mathbb{C}^{mn}\otimes \mathbb{C}^{n\ell})_{\ell mn}$, it can be defined as follows
\begin{align}\label{eq-flmn}
F_{(\ell,m,n)}=\widehat{\langle \lambda,\mu,\nu |}\left(\epsilon,\pi_{H}^{(2)},\pi_{H}^{(3)}\right)\mathcal{P}_{\ell mn},
\end{align}
and for $\omega\in\mathbb{C}^{\ell m}\otimes \mathbb{C}^{mn}\otimes \mathbb{C}^{n\ell}$ we have
\begin{align}\label{eq-flmnw0}
F_{(\ell,m,n)}(\omega)=\operatorname{val}_{H}(\omega^{\otimes\ell mn}).
\end{align}
So Remark 5.14  of \cite{BI17} can be restated as the following theorem.
\begin{theorem}\cite[Remark 5.14]{BI17}\label{thm-lmn}
For $F_{(\ell,m,n)}$ defined in (\ref{eq-flmn}),
we have
$$F_{(\ell,m,n)} ((g_1, g_2, g_3)w)=det(g_1)^{n}det(g_2)^{\ell} det(g_3)^{m} F_{(\ell,m,n)}(w)$$
for all $g_1\in \operatorname{GL}_{\ell m}, g_2\in \operatorname{GL}_{mn}, g_3\in \operatorname{GL}_{n\ell}$ and $w\in \mathbb{C}^{\ell m} \otimes \mathbb{C}^{mn}\otimes\mathbb{C}^{n\ell}$. Moreover, $F_{(\ell,m,n)}\neq0.$
\end{theorem}
The proof of Theorem \ref{thm-lmn} follows from the proof of Theorem 5.13 of \cite{BI17} straightforwardly. Thus, just like $F_{n}\neq0$, the proof of the assertion $F_{(\ell,m,n)}\neq0$ is also indirect. So Problem \ref{prob-fnn0} can be generalized to the following problem.
\begin{problem}\label{prob-flmnn0-sec4}
Give a direct proof of $F_{(\ell,m,n)}\neq0$ by evaluating $F_{(\ell,m,n)}\neq0$  at a (generic) $w\in \mathbb{C}^{\ell m} \otimes \mathbb{C}^{mn}\otimes\mathbb{C}^{n\ell}$.
\end{problem}
\begin{remark}\label{re-genric}
By the discussion in Section 2 of \cite{OR20}, except some special cases, under the action of  $G=\operatorname{GL}_{\ell m}\times \operatorname{GL}_{mn}\times \operatorname{GL}_{n\ell}$, the space $\mathbb{C}^{\ell m}\otimes \mathbb{C}^{mn}\otimes \mathbb{C}^{n\ell}$ is partitioned into infinite many group orbits. Moreover, by Remark 5.5.1.3 of \cite{Lan12}, except $\langle2,2,2\rangle$, $\langle n,n,n\rangle$ is far from being generic. However, by Proposition 4.2 and 5.2 of \cite{BI11}, we know that $\langle n,n,n \rangle$ and $\langle n^2 \rangle$ are $\operatorname{SL}_{n^2}^{3}$-stable. From the discussion of Section 4 of \cite{BGO+17}, this means that $0\notin\overline{\operatorname{SL}_{n^2}^{3}\langle n,n,n \rangle}$ and $\overline{\operatorname{SL}_{n^2}^{3}\langle n^2 \rangle}$, the orbit closures of
$\langle n,n,n \rangle$ and $\langle n^2 \rangle$. So by Theorem 2.1 of \cite{BLM+21}, there exist invariants $F_1$,
$F_2\in\mathcal{O}\left(\otimes^3\mathbb{C}^{n^2}\right)^
{\operatorname{SL}_{n^2}^3}_{n^2\delta}$ such that
$F_1(\langle n,n,n \rangle)$, $F_2(\langle n^2 \rangle)\neq0$.
\end{remark}

Setting $n=1$, we have the following proposition whose proof is also indirect.
\begin{proposition}
 $F_{(\ell,m,1)} (\langle \ell,m,1\rangle)\neq0$.
\end{proposition}
\begin{proof}
Since $\langle \ell,m,1\rangle\in \mathbb{C}^{\ell m}\otimes \mathbb{C}^{m}\otimes\mathbb{C}^{\ell}$ and the tensor rank of $\langle \ell,m,1\rangle$ is $\ell m$ \cite{Lan17},
by Remark 2.8 of \cite{OR20} the orbit $G\cdot\langle \ell,m,1\rangle$ is Zariski-dense, where $G=\operatorname{GL}_{\ell m}\times \operatorname{GL}_{m}\times \operatorname{GL}_{\ell}$. If $F_{(\ell,m,1)} (\langle \ell,m,1\rangle)=0$, then by Theorem \ref{thm-lmn} we have
$$F_{(\ell,m,1)} ((g_1\otimes g_2\otimes g_3)\langle \ell,m,1\rangle)=det(g_1)det(g_2)^{\ell} det(g_3)^{m} F_{(\ell,m,1)}(\langle \ell,m,1\rangle)=0$$
for all $g_1\in \operatorname{GL}_{\ell m},~ g_2\in \operatorname{GL}_{m},~ g_3\in \operatorname{GL}_{\ell}$.
Thus, $F_{(\ell,m,1)}=0$ which is a contradiction.
\end{proof}

Now, we provide some conditions when $F_{(\ell,m,n)}(\omega)=0$.
In Section 4.3 of \cite{BI13}, B\"{u}rgisser and Ikenmeyer obtained a vanishing condition for $F_H$ by the chromatic index of $H$. In the following, we give another vanishing condition.

Suppose that $w=\sum_{i=1}^r w^{(1)}_i\otimes w^{(2)}_i\otimes w^{(3)}_i\in \mathbb{C}^{d_1}\otimes \mathbb{C}^{d_2}\otimes \mathbb{C}^{d_3}$. Let $W_1\subseteq \mathbb{C}^{d_1}$ be the
subspace spanned by $w^{(1)}_i$ ($i=1,2,...,r$). $W_2$ and $W_3$ are defined similarly. The dimension of $W_i$ ($i=1,2,3$) is denoted by $\dim W_i$. Then we have the following proposition.
\begin{proposition}\label{prop-vanid1d2d3}
Let $H$ be an obstruction design of type ($\lambda,\mu,\nu$) for $\lambda,\mu,\nu\vdash d$.
Suppose that $F_H\in \operatorname{Sym}^d(\mathbb{C}^{d_1}\otimes \mathbb{C}^{d_2}\otimes \mathbb{C}^{d_3})^*$ is the highest weight vector corresponding to $H$.
For tensors $w\in \mathbb{C}^{d_1}\otimes \mathbb{C}^{d_2}\otimes \mathbb{C}^{d_3}$, if one of the three conditions happens
\begin{align}\label{eq-condw1w2w3}
 {\rm(\romannumeral1)}~\dim W_1< \ell(\lambda),\quad{\rm(\romannumeral2)}~\dim W_2<\ell(\mu),\quad{\rm (\romannumeral3)}~\dim W_3<\ell(\nu),
\end{align}
then $F_H(w)=0$.
\end{proposition}
\begin{proof}
Suppose that $w=\sum_{i=1}^r w^{(1)}_i\otimes w^{(2)}_i\otimes w^{(3)}_i\in \mathbb{C}^{d_1}\otimes \mathbb{C}^{d_2}\otimes \mathbb{C}^{d_3}$.
Set $w_i=w^{(1)}_i\otimes w^{(2)}_i\otimes w^{(3)}_i$ for $i=1,2,..,r$.  Let $I:[d]\to [r]$ be a map
and $w_I=w_{I(1)}\otimes w_{I(2)}\cdots\otimes w_{I(d)}\in \otimes^d(\mathbb{C}^{d_1}\otimes \mathbb{C}^{d_2}\otimes \mathbb{C}^{d_3})$ be the corresponding triple labeling.
For $j=1,2,3$, let $w^{(j)}_{I}=\otimes_{i=1}^d w_{I(i)}^{(j)}\in\otimes^d \mathbb{C}^{d_j}$.

With out loss of generalization, suppose that $\dim W_1< \ell(\lambda)={}^{t}\lambda_1$. Then if we choose any
${}^{t}\lambda_1$ vectors in $\{w^{(1)}_i|i=1,2,...,r\}$, they are all linear dependent.
So we have
$$\langle\hat{\lambda}|\pi_{H}^{(1)}|w^{(1)}_{I}\rangle=0$$
and therefore
\begin{align*}
\operatorname{val}_H(w_I)=\widehat{\langle\lambda,\mu,\nu|}(\pi_{H}^{(1)},
\pi_{H}^{(2)},\pi_{H}^{(3)})|w_I\rangle=
\widehat{\langle\lambda|}\pi_{H}^{(1)}|w^{(1)}_{I}\rangle \widehat{\langle\mu|}\pi_{H}^{(2)}|w^{(2)}_{I}\rangle
\widehat{\langle\nu|}\pi_{H}^{(3)}|w^{(3)}_{I}\rangle=0.
\end{align*} Thus, $F_H(w)=0$.
\end{proof}
Suppose that $\lambda=\ell m\times n$, $\mu=mn\times\ell$ and $\nu=n\ell\times m$.  In Proposition \ref{prop-vanid1d2d3}, let $H=B(n,\ell,m)$ which is the unique obstruction design of type ($\lambda,\mu,\nu$). If we let $d_1=\ell m$,
$d_2=mn$ and $d_3=n\ell$, then we have the following corollary.
\begin{corollary}\label{cor-lmnvan3}
Suppose that $H=B(n,\ell,m)$. Then for any tensor $w\in\mathbb{C}^{\ell m} \otimes \mathbb{C}^{mn}\otimes\mathbb{C}^{n\ell}$, if one of the three conditions happens
\begin{align*}
 {\rm(\romannumeral1)}~\dim W_1< \ell m,\quad{\rm(\romannumeral2)}~\dim W_2<mn,\quad{\rm (\romannumeral3)}~\dim W_3<n\ell,
\end{align*}
then $F_{(\ell, m,n)}(w)=0$.
\end{corollary}

For $a,b\in \mathbb{N}$, let $\{E_{i,j}^{a\times b}|i=1,\ldots,a;~j=1,\ldots,b\}$ be the standard basis of $\mathbb{C}^{a\times b}$, i.e. the $(i,j)$ entry of $E_{i,j}^{a\times b}$ is 1 and all other entries are zeros.
Suppose that $\ell'\leq \ell$, $m'\leq m$ and $n'\leq n$. Considering $E_{i,j}^{\ell'\times m'}$, $E_{j,k}^{m'\times n'}$ and $E_{k,i}^{n'\times\ell'}$ as $E_{i,j}^{\ell\times m}$, $E_{j,k}^{m\times n}$ and $E_{k,i}^{n\times\ell}$ respectively, we can embed the matrix multiplication tensor $\langle \ell',m',n'\rangle \in \mathbb{C}^{\ell' m'}\otimes \mathbb{C}^{m'n'}\otimes \mathbb{C}^{n'\ell'}$ into
$\mathbb{C}^{\ell m}\otimes \mathbb{C}^{mn}\otimes \mathbb{C}^{n\ell}$. By this embedding and Corollary \ref{cor-lmnvan3}, we have the following corollary.
\begin{corollary}\label{cor-mmtlmn}
Suppose that $\ell'\leq \ell$, $m'\leq m$ and $n'\leq n$. For the matrix multiplication tensor $\langle \ell',m',n'\rangle$, if one of the three conditions happens
\begin{align*}
 {\rm(\romannumeral1)}~\ell'<\ell,\quad{\rm(\romannumeral2)}~m'<m,\quad{\rm (\romannumeral3)}~n'<n,
\end{align*}
then $F_{(\ell, m,n)}(\langle \ell',m',n'\rangle)=0$.
\end{corollary}
In Corollary \ref{cor-mmtlmn}, if we let $\ell=m=n=3$ and $\ell'=2$, $m'=3$, $n'=3$, we have that $F_{3}(\langle 2,3,3\rangle)=0$. By Theorem 1.4 of \cite{CHLan19}, we know that
the border rank of $\langle 2,3,3\rangle$ is 14. On the other hand, the chromatic index of $H=B(3,3,3)$ is 9, so in some sense Proposition \ref{prop-vanid1d2d3} is better than Proposition 4.2 of \cite{BI13}.

In the following, we make a coarse description of the evaluation $F_{(\ell,m,n)}(\omega)$, which extends the discussion in Section \ref{se:prelim}. Let $H=B(n,\ell,m)\simeq [\ell mn]$. Suppose that $\omega\in \mathbb{C}^{\ell m}\otimes\mathbb{C}^{mn}\otimes\mathbb{C}^{n\ell}$ is decomposed into distinct rank one tensors as $\omega=\sum_{i=1}^r \omega^{(1)}_i \otimes \omega^{(2)}_i\otimes \omega^{(3)}_i=\sum_{i=1}^r \omega_i$. Then we have
\begin{align}\label{eq-wlmn}
\omega^{\otimes \ell mn}=\sum_{I:[\ell mn]\to [r]}\omega_I
\end{align}
where the sum is taken over all maps $I:[\ell mn]\to [r]$ and $$\omega_I=\omega_{I(1)}\otimes \omega_{I(2)}\cdots\otimes \omega_{I(\ell mn)}\in \otimes^{\ell mn}(\mathbb{C}^{\ell m}\otimes\mathbb{C}^{mn}\otimes\mathbb{C}^{n\ell}).$$
As in Section \ref{se:prelim}, in the following we often use the following notation.
\begin{notation}
With symbols above, for $\omega_I=\omega_{I(1)}\otimes \omega_{I(2)}\cdots\otimes \omega_{I(\ell mn)}\in \otimes^{\ell mn}(\mathbb{C}^{\ell m}\otimes\mathbb{C}^{mn}\otimes\mathbb{C}^{n\ell})$ where $\omega_{I(i)}=\omega_{I(i)}^{(1)}\otimes \omega_{I(i)}^{(2)}\otimes \omega_{I(i)}^{(3)}$, if we identify $\omega_I$ as
$$\omega_{I}=(\omega_{I(1)}, \omega_{I(2)},\cdots,\omega_{I(\ell mn)})=\left(
  \begin{array}{cccc}
    \omega_{I(1)}^{(1)} & \omega_{I(2)}^{(1)} & \cdots & \omega_{I(\ell mn)}^{(1)} \\
    \omega_{I(1)}^{(2)} & \omega_{I(2)}^{(2)} & \cdots & \omega_{I(\ell mn)}^{(2)} \\
    \omega_{I(1)}^{(3)} & \omega_{I(2)}^{(3)} & \cdots & \omega_{I(\ell mn)}^{(3)} \\
  \end{array}
\right),$$
then it can be considered as a triple labeling of $H=B(n,\ell,m)$.
When the decomposition $\omega=\sum_{i=1}^r \omega^{(1)}_i \otimes \omega^{(2)}_i\otimes \omega^{(3)}_i=\sum_{i=1}^r \omega_i$ is fixed, for convenience sometimes we will identify $\omega_I$ as $I$.
\end{notation}

Let $I$ and $I'$ be two maps from $[\ell mn]$ to
$[r]$. Then in the summands of (\ref{eq-wlmn}), we say that $\omega_I=\omega_{I(1)}\otimes \omega_{I(2)}\cdots\otimes \omega_{I(\ell mn)}$ and $\omega_{I'}=\omega_{I'(1)}\otimes \omega_{I'(2)}\cdots\otimes \omega_{I'(\ell mn)}$
are equivalent if the following two sets that allow repetition
are equal
$$\{\omega_{I(1)},\omega_{I(2)},...,\omega_{I(\ell mn)}\}=\{\omega_{I'(1)}, \omega_{I(2)},...,\omega_{I'(\ell mn)}\}.$$
Let $[\omega_{I_0}]\subseteq \{\omega_I|I:[\ell mn]\to [r]\}$ denote the set of all triple labelings that are equivalent to $\omega_{I_0}$. Define
$\pi\omega_I:=\omega_{I(\pi^{-1}(1))}\otimes \omega_{I(\pi^{-1}(2))}\cdots\otimes \omega_{I(\pi^{-1}(\ell mn))}\in \otimes^{\ell mn}(\mathbb{C}^{\ell m}\otimes\mathbb{C}^{ mm}\otimes\mathbb{C}^{n\ell} )$ for all $I$. Then we have $[\omega_{I_0}]=\{\pi\omega_{I_0}|\pi\in S_{\ell mn}\}$.
Let $\tilde{V}(\omega,H)$ denote the set of all equivalent classes in the summands of (\ref{eq-wlmn}).
Then (\ref{eq-wlmn}) can also be written as
\begin{align}\label{eq-wlmneq}
\omega^{\otimes \ell mn}=\sum_{I:[\ell mn]\to [r]}\omega_I=
\sum_{[\omega_I]\in\tilde{V}(\omega,H)}\sum_{\pi\in S_{\ell mn}}\pi \omega_I=\sum_{[\omega_I]\in \tilde{V}(\omega,H)}\mathcal{P}_{\ell mn} \omega_I.
\end{align}
An equivalent class $[\omega_{I}]\in\tilde{V}(\omega,H)$ is called a \emph{valid equivalent class}, if
there exists a $\omega_{I'}=\pi\omega_{I}\in[\omega_{I}]$ such that $\operatorname{val}_H(\omega_{I'})\neq0$. In this case, $\omega_{I'}$ is called a \emph{valid triple labeling}. Let $V(\omega,H)\subseteq \tilde{V}(\omega,H)$ denote the subset of all
valid equivalent classes.
Then by (\ref{eq-flmnw0}) and (\ref{eq-wlmneq}) we have
\begin{align*}
F_{(\ell,m,n)}(\omega)=\sum_{[\omega_I]\in V(\omega,H)}\operatorname{val}_{H}\left(\mathcal{P}_{\ell mn} \omega_I\right).
\end{align*}

For $v=\otimes_{i=1}^{\ell mn}(v_{i}^{(1)}\otimes v_{i}^{(2)}\otimes v_{i}^{(3)})\in\otimes^{\ell mn}(\mathbb{C}^{\ell m}\otimes\mathbb{C}^{mn}\otimes\mathbb{C}^{n\ell})$, define $F_{(\ell,m,n)}[v]$ as
\begin{align}\label{eq-flmn[]}
F_{(\ell,m,n)}[v]:=\operatorname{val}_{H}\left(\mathcal{P}_{\ell mn} v\right)=
\sum_{\pi\in S_{\ell mn}}\operatorname{val}_{H}(\pi v).
\end{align}
Then the evaluation $F_{(\ell,m,n)}(\omega)$ can be written as
\begin{align}\label{eq-flmnw}
F_{(\ell,m,n)}(\omega)=\sum_{[\omega_I]\in V(\omega,H)}\operatorname{val}_{H}\left(\mathcal{P}_{\ell mn} \omega_I\right)=\sum_{[\omega_I]\in V(\omega,H)}F_{(\ell,m,n)}[\omega_I].
\end{align}

\subsection{The evaluation of $F_{(\ell,m,n)}(\langle\ell,m,n\rangle)$}\label{subsec-evalflmn}
\

For $a\in \mathbb{N}$, let $\{e_{i}^{(a)}|i=1,\ldots,a\}$ be the standard basis of $\mathbb{C}^{a}$, i.e. the $i$th entry of $e_{i}^{(a)}$ is 1 and all other entries are zeros. In literatures (see e. g. \cite[Example 2.1.2.3]{Lan17} and \cite[Def. 4.2]{OR20}), $\langle\ell,m,n\rangle$ is usually defined as
$$\langle\ell,m,n\rangle=\sum_{i=1}^{\ell}\sum_{j=1}^{m}\sum_{k=1}^n
E_{i,j}^{\ell\times m}\otimes E_{j,k}^{m\times n}\otimes E_{k,i}^{n\times\ell}
\in\mathbb{C}^{\ell\times m}\otimes \mathbb{C}^{m\times n}\otimes \mathbb{C}^{n\times\ell}.$$
There is an isomorphism $\phi_{a,b}:\mathbb{C}^{a\times b}\to \mathbb{C}^{ab}$, which is defined by
\begin{equation}\label{eq-eijab}
E_{i,j}^{a\times b}\mapsto e^{(ab)}_{(i-1)b+j},~i=1,...,a;~j=1,...,b.
\end{equation}
So by (\ref{eq-eijab}), $\langle\ell,m,n\rangle$ can be identified as
\begin{align}\label{eq-lmn}
\langle\ell,m,n\rangle&=\sum_{i=1}^{\ell}\sum_{j=1}^{m}\sum_{k=1}^n
e_{(i-1)m+j}^{(\ell m)}\otimes e_{(j-1)n+k}^{(mn)}\otimes e_{(k-1)\ell+i}^{(n\ell)}\\
&=\sum_{i=1}^{\ell}\sum_{j=1}^{m}\sum_{k=1}^n
\langle\ell,m,n\rangle_{(i,j,k)}
\in\mathbb{C}^{\ell  m}\otimes \mathbb{C}^{m  n}\otimes \mathbb{C}^{n \ell},\notag
\end{align}
where
$e_{(i-1)m+j}^{(\ell m)}\otimes e_{(j-1)n+k}^{(mn)}\otimes e_{(k-1)\ell+i}^{(n\ell)}$ is denoted by $\langle\ell,m,n\rangle_{(i,j,k)}$.
When doing evaluation, in this paper, we use (\ref{eq-lmn}) as the definition of $\langle\ell,m,n\rangle$.

\begin{example}
By (\ref{eq-lmn}), we have
\begin{align}\label{eq-223}
 \langle 2,2,3\rangle&=e^{(4)}_{1}\otimes e^{(6)}_{1}\otimes e^{(6)}_{1}+e^{(4)}_{1}\otimes e^{(6)}_{2}\otimes e^{(6)}_{3}+
e^{(4)}_{1}\otimes e^{(6)}_{3}\otimes e^{(6)}_{5}+e^{(4)}_{2}\otimes e^{(6)}_{4}\otimes e^{(6)}_{1}\notag\\
&+e^{(4)}_{2}\otimes e^{(6)}_{5}\otimes e^{(6)}_{3}+e^{(4)}_{2}\otimes e^{(6)}_{6}\otimes e^{(6)}_{5}+e^{(4)}_{3}\otimes e^{(6)}_{1}\otimes e^{(6)}_{2}\notag\\
&+e^{(4)}_{3}\otimes e^{(6)}_{2}\otimes e^{(6)}_{4}+e^{(4)}_{3}\otimes e^{(6)}_{3}\otimes e^{(6)}_{6}+e^{(4)}_{4}\otimes e^{(6)}_{4}\otimes e^{(6)}_{2}\notag\\
&+e^{(4)}_{4}\otimes e^{(6)}_{5}\otimes e^{(6)}_{4}+e^{(4)}_{4}\otimes e^{(6)}_{6}\otimes e^{(6)}_{6}\in\mathbb{C}^{4}\otimes \mathbb{C}^{6}\otimes \mathbb{C}^{6}.
\end{align}
\end{example}

In the following, for simplicity,  $\left(e_{(i-1)m+j}^{(\ell m)}, e_{(j-1)n+k}^{(mn)}, e_{(k-1)\ell+i}^{(n\ell)}\right)^t$ is denoted by
\begin{align}\label{eq-ijklmn}
\left(\begin{array}{c}
(i-1)m+j \\
(j-1)n+k \\
(k-1)\ell+i \\
\end{array}
\right).
\end{align}
So the set of summands of $\langle\ell,m,n\rangle$ in (\ref{eq-lmn}) can be denoted by
\begin{align}\label{eq-tlmn}
T(\langle\ell,m,n\rangle):=\left\{\left(
                                \begin{array}{c}
                                  (i-1)m+j \\
                                  (j-1)n+k \\
                                 (k-1)\ell+i\\
                                \end{array}
                              \right)   \bigg|~i\in[\ell],~j\in [m]~
                              k\in[n]
 \right\}.
\end{align}

Recall that $H=B(n,\ell,m)$. By (\ref{eq-flmnw}),
we have
\begin{align}\label{}
F_{(\ell,m,n)} (\langle\ell,m,n\rangle)=\sum_{[\langle\ell,m,n\rangle_I]\in V(\langle\ell,m,n\rangle, H)}F_{(\ell,m,n)}[\langle\ell,m,n\rangle_I],
\end{align}
where $\langle\ell,m,n\rangle_I:H\simeq [\ell mn]\to T(\langle\ell,m,n\rangle)\subseteq \mathbb{C}^{\ell  m}\times \mathbb{C}^{m  n}\times \mathbb{C}^{n \ell}$
is a triple labeling on $H$. So the evaluation of $F_{(\ell,m,n)} (\langle\ell,m,n\rangle)$ can be divided into two steps:
\begin{enumerate}
  \item\label{it-st1} Firstly, search the set of all valid equivalent classes $V(\langle\ell,m,n\rangle,H)$;
 \item\label{it-st2} Secondly, for each $[\langle\ell,m,n\rangle_I]\in V(\langle\ell,m,n\rangle,H)$, evaluate $F_{(\ell,m,n)}[\langle\ell,m,n\rangle_I]$. Then, sum up them.
\end{enumerate}

Next, we give a detailed discussion of (\ref{it-st1}) and (\ref{it-st2}) above.

(\ref{it-st1}) The way to search all valid equivalent classes $V(\langle\ell,m,n\rangle,H)$ can be done as follows.

The main task is to characterize when $\operatorname{val}_{H}(\langle\ell,m,n\rangle_I)\neq0$, where  $I$ is the triple labeling from $H=B(n,\ell,m)$ to the summands of $\langle\ell,m,n\rangle$.

Firstly, we write $\langle\ell,m,n\rangle_I$ in a way which is convenient for the evaluation.
A triple labeling  $\langle\ell,m,n\rangle_I$ from $H=B(n,\ell,m)$ to the summands of $\langle\ell,m,n\rangle$ in (\ref{eq-lmn}) can be expressed as
\begin{align}\label{eq-lmni}
\langle\ell,m,n\rangle_I:=\left(I(1),I(2),...,I(\ell mn)\right)=\left(
\begin{array}{cccc}
e^{(\ell m)}_{I(1)}  & e^{(\ell m)}_{I(2)} & \cdots & e^{(\ell m)}_{I(\ell mn)}\\
e^{(mn)}_{I(1)} & e^{(mn)}_{I(2)} & \cdots & e^{(mn)}_{I(\ell mn)} \\
e^{(n\ell)}_{I(1)} & e^{(n\ell)}_{I(2)} & \cdots & e^{(n\ell)}_{I(\ell mn)}\\
\end{array}
\right),
\end{align}
where we write $\langle\ell,m,n\rangle_{I(i)}$ as $I(i)$ and
\begin{align*}
e^{(\ell m)}_{I(1)}\otimes e^{(mn)}_{I(1)} \otimes e^{(n\ell)}_{I(1)}=e_{(i_1-1)m+j_1}^{(\ell m)}\otimes e_{(j_1-1)n+k_1}^{(mn)}\otimes e_{(k_1-1)\ell+i_1}^{(n\ell)}
\end{align*}
 for some $i_1,j_1,k_1$, and similarly for $e^{(\ell m)}_{I(2)}\otimes e^{(mn)}_{I(2)} \otimes e^{(n\ell)}_{I(2)}$,...,etc. That is, under the lexicographic ordering, the $j$th point of $B(n,\ell,m)$ is mapped to the $j$th column  of (\ref{eq-lmni}). So under (\ref{eq-ijklmn}), (\ref{eq-lmni})
can be written as
\begin{align}\label{eq-abvlmni}
 \langle\ell,m,n\rangle_I=\left(
\begin{array}{cccc}
(i_1-1)m+j_1  & (i_2-1)m+j_2 & \cdots &  (i_{\ell mn}-1)m+j_{\ell mn}\\
(j_1-1)n+k_1 & (j_2-1)n+k_2 & \cdots & (j_{\ell mn}-1)n+k_{\ell mn} \\
(k_1-1)\ell+i_1 & (k_2-1)\ell+i_2 & \cdots & (k_{\ell mn}-1)\ell+i_{\ell mn}\\
\end{array}
\right).
\end{align}

Generally, suppose that  $v:H\simeq [\ell mn]\to \mathbb{C}^{\ell m}\times \mathbb{C}^{m n}\times \mathbb{C}^{n\ell}$ is a triple labeling. Let $v(i)=(v^{(1)}_i,v^{(2)}_i,v^{(3)}_i)^t$
where $v^{(1)}_i\in \mathbb{C}^{\ell m}$, $v^{(2)}_i\in \mathbb{C}^{mn}$, $v^{(3)}_i\in \mathbb{C}^{n\ell}$ and $i\in [\ell mn]$. So we can write $v$ as
\begin{align*}
v=(v(1),v(2),...,v({\ell mn}))=\left(
  \begin{array}{cccc}
    v^{(1)}_1, & v^{(1)}_2, & \cdots ,& v^{(1)}_{\ell mn} \\
       v^{(2)}_1, & v^{(2)}_2, & \cdots ,& v^{(2)}_{\ell mn} \\
           v^{(3)}_1, & v^{(3)}_2, & \cdots ,& v^{(3)}_{\ell mn}
\end{array}
\right).
\end{align*}
As in (\ref{eq-valh}),  we can define $\operatorname{val}_{H}(v)$ as follows:
\begin{align}\label{eq-valmn}
   \operatorname{val}_{H}(v)=&\prod_{\textbf{e}_{i}^{(1)}\in E^{(1)}} \operatorname{val} \left(v_{s_1}^{(1)},
v_{s_2}^{(1)},\cdots,v_{s_{|\textbf{e}_{i}^{(1)}|}}^{(1)}\right)\times\notag\\
   &\prod_{\textbf{e}_{j}^{(2)}\in E^{(2)}} \operatorname{val} \left(v_{s_1}^{(2)},
v_{s_2}^{(2)},\cdots,v_{s_{|\textbf{e}_{j}^{(2)}|}}^{(2)}\right)\times\\
&\prod_{\textbf{e}_{k}^{(3)}\in E^{(3)}} \operatorname{val} \left(v_{s_1}^{(3)},
v_{s_2}^{(3)},\cdots,v_{s_{|\textbf{e}_{k}^{(3)}|}}^{(3)}\right).\notag
\end{align}
Since $|\textbf{e}_{i}^{(1)}|=\ell m$ and
$v_{s_j}^{(1)}\in \mathbb{C}^{\ell m}$ in (\ref{eq-valmn}), we have
\begin{align}\label{eq-valdet}
\operatorname{val}\left(v_{s_1}^{(1)},
v_{s_2}^{(1)},\cdots,v_{s_{|\textbf{e}_{i}^{(1)}|}}^{(1)}\right)=\det \left(v_{s_1}^{(1)},v_{s_2}^{(1)},\cdots,v_{s_{|\textbf{e}_{i}^{(1)}|}}^{(1)}\right).
\end{align}
Similarly, we have
$\operatorname{val}(v_{s_1}^{(2)},
v_{s_2}^{(2)},\cdots,v_{s_{|\textbf{e}_{j}^{(2)}|}}^{(2)})=\det (v_{s_1}^{(2)},
v_{s_2}^{(2)},\cdots,v_{s_{|\textbf{e}_{j}^{(2)}|}}^{(2)})$ and
$\operatorname{val}(v_{s_1}^{(3)},
v_{s_2}^{(3)},\cdots,v_{s_{|\textbf{e}_{k}^{(3)}|}}^{(3)})=\det (v_{s_1}^{(3)},
v_{s_2}^{(3)},\cdots,v_{s_{|\textbf{e}_{k}^{(3)}|}}^{(3)})$.

Now, let $v=\langle\ell,m,n\rangle_I$ be the triple labeling defined in (\ref{eq-lmni}).
Then $v_{s_i}^{(1)}=e^{(\ell m)}_{t_i}$, $v_{s_j}^{(2)}=e^{(mn)}_{t_j}$ and $v_{s_k}^{(3)}=e^{(mn)}_{t_k}$ belong to the standard basis of
$\mathbb{C}^{\ell m}$, $\mathbb{C}^{m n}$ and $\mathbb{C}^{n\ell}$, respectively. So we have
\begin{align}\label{eq-detinv}
\det \left(v_{s_1}^{(1)},
v_{s_2}^{(1)},\cdots,v_{s_{|\textbf{e}_{i}^{(1)}|}}^{(1)}\right)=\det
\left(e_{t_1}^{(\ell m)},
e_{t_2}^{(\ell m)},\cdots,e_{t_{|\textbf{e}_{i}^{(1)}|}}^{(\ell m)}\right).
\end{align}

Thus, if $\det(e_{t_1}^{(\ell m)},
e_{t_2}^{(\ell m)},\cdots,e_{t_{|\textbf{e}_{i}^{(1)}|}}^{(\ell m)})\neq0$, for each $i\in [n]$ the sequence $(t_1,t_2,...,t_{|\textbf{e}_{i}^{(1)}|})$ should be a permutation of the set $\{1,2,...,\ell m\}$. Hence, suppose that $\operatorname{val}_H(\langle\ell,m,n\rangle_I)\neq 0$ and
$\langle\ell,m,n\rangle_I$ is expressed as in (\ref{eq-abvlmni}), which is a $3\times\ell mn$ integer matrix. Then in the first row of (\ref{eq-abvlmni}), each  $i\in[\ell m]$ appears exactly $n$ times. Similar, in the second row of (\ref{eq-abvlmni}), each  $j\in[mn]$ appears exactly $\ell$ times.
In the third row of (\ref{eq-abvlmni}), each  $k\in[n\ell]$ appears exactly $m$ times.
In summarize, we have a necessary condition for valid equivalent class.
\begin{proposition}\label{cl-times}
Suppose that $[\langle\ell,m,n\rangle_I]$ is a valid equivalent class and $\operatorname{val}_H(\langle\ell,m,n\rangle_I)\neq 0$. If we write $\langle\ell,m,n\rangle_I$ as the $3\times \ell mn$ matrix in (\ref{eq-abvlmni}), then in the first row of (\ref{eq-abvlmni}), for each $i\in[\ell m]$, $i$ appears $n$ times; in the second row, for each $j\in[mn]$, $j$ appears $\ell$ times; in the third row, for each $k\in[n\ell]$, $k$ appears $m$ times.
\end{proposition}

Thus, to search valid equivalent classes, we should construct matrices as in (\ref{eq-abvlmni}), which satisfy Proposition \ref{cl-times}. Under (\ref{eq-abvlmni}), two triple labelings $\langle\ell,m,n\rangle_I$ and $\langle\ell,m,n\rangle_{I'}$ are in the same equivalent class if they can be transformed to each other by permuting columns.
By Proposition \ref{cl-times}, we can search valid equivalent classes by computer. For example, we have searched that there are 7 valid equivalent classes in $V(\langle2,2,3\rangle,H)$.
In the following example, we give an interesting valid equivalent class $[\langle\ell,m,n\rangle_{I_0}]$.
\begin{example}\label{ex-typlab}
Recall that $T(\langle\ell,m,n\rangle)$ is set of summands
of $\langle\ell,m,n\rangle$ defined in (\ref{eq-tlmn}).
For $(k,i,j)\in B(n,\ell,m)$, define $\langle\ell,m,n\rangle_{I_0}:B(n,\ell,m)\to T(\langle\ell,m,n\rangle)$ by
$$ (k,i,j)\mapsto\left(\begin{array}{c}
(i-1)m+j \\
(j-1)n+k \\
(k-1)\ell+i \\
\end{array}
\right),$$
where  $i\in [\ell]$, $j\in [m]$ and $k\in [n]$.
In other words, the triple labeling is obtained by putting
$\langle\ell,m,n\rangle_{i,j,k}$ in the position $(k,i,j)$ of $B(n,\ell,m)$. So, $I_0$ is a bijection. By the correspondence between $\langle\ell,m,n\rangle_{(i,j,k)}$ and $E_{i,j}^{\ell\times m}\otimes E_{j,k}^{m\times n}\otimes E_{k,i}^{n\times\ell}$ defined in
(\ref{eq-eijab}), we can easily check that $\operatorname{val}_H(\langle\ell,m,n\rangle_{I_0})\neq0$. So $[\langle\ell,m,n\rangle_{I_0}]$ is a valid equivalent class.

For example, by (\ref{eq-223}) and (\ref{eq-abvlmni}), we can express $\langle2,2,3\rangle_{I_0}$ in the following way
\begin{align}\label{eq-abv223}
\langle2,2,3\rangle_{I_0}=
\left(\begin{array}{cccccccccccc}
1 & 2 & 3 & 4 & 1 & 2 & 3 & 4 & 1 & 2 & 3 & 4 \\
1 & 4 & 1 & 4 & 2 & 5 & 2 & 5 & 3 & 6 & 3 & 6 \\
1 & 1 & 2 & 2 & 3 & 3 & 4 & 4 & 5 & 5 & 6 & 6 \\
\end{array}
\right).
\end{align}
\end{example}

(\ref{it-st2}) Now, suppose that $[\langle\ell,m,n\rangle_I]$ is a valid equivalent class. By (\ref{eq-flmn[]}), we have
$$F_{(\ell,m,n)}[\langle\ell,m,n\rangle_I]=\sum_{\pi\in S_{\ell mn}}\operatorname{val}_{H}(\pi \langle\ell,m,n\rangle_I).$$
Observe that
$$\det
\left(e_{t_1}^{(\ell m)},
e_{t_2}^{(\ell m)},\cdots,e_{t_{|\textbf{e}_{i}^{(1)}|}}^{(\ell m)}\right)=(-1)^{inv\left(t_1,t_2,...,t_{|\textbf{e}_{i}^{(1)}|}\right)},$$
where $inv(t_1,t_2,...,t_{|\textbf{e}_{i}^{(1)}|})$ denotes the number of inversions of $(t_1,t_2,...,t_{|\textbf{e}_{i}^{(1)}|})$.
Similar results hold for
$\det (v_{s_1}^{(2)},
v_{s_2}^{(2)},\cdots,v_{s_{|\textbf{e}_{j}^{(2)}|}}^{(2)})$ and
$\det(v_{s_1}^{(3)},v_{s_2}^{(3)},\cdots,v_{s_{|\textbf{e}_{k}^{(3)}|}}^{(3)})$.
Thus, by (\ref{eq-valmn}), (\ref{eq-valdet}) and (\ref{eq-detinv}) if $\operatorname{val}_{H}(\langle\ell,m,n\rangle_I)\neq0$, then it equals to the products of 1 and -1 and therefore $\operatorname{val}_{H}(\langle\ell,m,n\rangle_I)=1$ or $-1$. Hence, $F_{(\ell,m,n)}[\langle\ell,m,n\rangle_I]$ is the sums of -1, 0 and 1.
In the following, we give some examples verified by computer.
\begin{example}\label{exm-223}

Let $H=B(3,2,2)=\{(i,j,k)~|i\in[3],~j\in[2],~k\in[2]\}$. Then we make the following identification:
\begin{align*}
 \begin{array}{cccc}
  (1,1,1)=1  & (1,1,2)=2  &  (1,2,1)=3  & (1,2,2)=4\\
  (2,1,1)=5  & (2,1,2)=6  &  (2,2,1)=7  & (2,2,2)=8 \\
  (3,1,1)=9  & (3,1,2)=10 &  (3,2,1)=11 & (3,2,2)=12
\end{array}
\end{align*}

In this way, we have: $E^{(1)}=\{\textbf{e}_1^{(1)},\textbf{e}_2^{(1)},\textbf{e}_3^{(1)}\}$ where
$\textbf{e}_1^{(1)}=\{1,2,3,4\}$, $\textbf{e}_2^{(1)}=\{5,6,7,8\}$ and
$\textbf{e}_3^{(1)}=\{9,10,11,12\}$; $E^{(2)}=\{\textbf{e}_1^{(2)},\textbf{e}_2^{(2)}\}$ where
$\textbf{e}_1^{(2)}=\{1,2,5,6,9,10\}$ and $\textbf{e}_2^{(2)}=\{3,4,7,8,11,12\}$;
$E^{(3)}=\{\textbf{e}_1^{(3)},\textbf{e}_2^{(3)}\}$ where
$\textbf{e}_1^{(3)}=\{1,3,5,7,9,11\}$ and $\textbf{e}_2^{(3)}=\{2,4,6,8,10,12\}$.

For any triple labeling $w:~H\simeq[12]\to \mathbb{C}^4\times \mathbb{C}^6\times \mathbb{C}^6$, let $w(i)=(w^{(1)}_i,w^{(2)}_i,w^{(3)}_i)^t$ for $i\in [12]$. Then
 $w$ can also be written as
\begin{align*}
w=\left(w(1),w(2),...,w(12)\right)=\left(
  \begin{array}{cccc}
    w^{(1)}_1, & w^{(1)}_2, & \cdots ,& w^{(1)}_{12} \\
       w^{(2)}_1, & w^{(2)}_2, & \cdots ,& w^{(2)}_{12} \\
           w^{(3)}_1, & w^{(3)}_2, & \cdots ,& w^{(3)}_{12}
\end{array}
\right).\end{align*}
Then we have
\begin{align*}
\operatorname{val}_{H}(w)=&\det\left(w^{(1)}_1, w^{(1)}_2,w^{(1)}_3, w^{(1)}_4\right)\cdot
\det\left(w^{(1)}_5, w^{(1)}_6,w^{(1)}_7, w^{(1)}_8\right)\cdot
\det\left(w^{(1)}_9, w^{(1)}_{10},w^{(1)}_{11}, w^{(1)}_{12}\right)
\times\notag\\
   &\det\left(w^{(2)}_1, w^{2)}_2,w^{(2)}_5, w^{(2)}_6,w^{(2)}_9, w^{(2)}_{10}\right)\cdot
\det\left(w^{(2)}_3, w^{(2)}_4,w^{(2)}_7, w^{(2)}_8,w^{(2)}_{11}, w^{(2)}_{12}\right)\times\\
&\det\left(w^{(3)}_1, w^{3)}_3,w^{(3)}_5, w^{(3)}_7,w^{(3)}_9, w^{(3)}_{11}\right)\cdot
\det\left(w^{(3)}_2, w^{(3)}_4,w^{(3)}_6, w^{(3)}_8,w^{(3)}_{10}, w^{(3)}_{12}\right).\notag
\end{align*}
Specially,
let $\langle2,2,3\rangle_{I_0}$
be the triple labeling defined in (\ref{eq-abv223}).
Then we have
\begin{align*}
\operatorname{val}_{H}(\langle2,2,3\rangle_{I_0})=&\det\left(e^{(4)}_1, e^{(4)}_2,e^{(4)}_3, e^{(4)}_4\right)\cdot
\det\left(e^{(4)}_1, e^{(4)}_2,e^{(4)}_3, e^{(4)}_4\right)\cdot
\det\left(e^{(4)}_1, e^{(4)}_2,e^{(4)}_3, e^{(4)}_4\right)
\times\\
   &\det\left(e^{(6)}_1, e^{6)}_4,e^{(6)}_2, e^{(6)}_5,e^{(6)}_3, e^{(6)}_{6}\right)\cdot
\det\left(e^{(6)}_1, e^{6)}_4,e^{(6)}_2, e^{(6)}_5,e^{(6)}_3, e^{(6)}_{6}\right)\times\\
&\det\left(e^{(6)}_1, e^{6)}_2,e^{(6)}_3, e^{(6)}_4,e^{(6)}_5, e^{(6)}_{6}\right)\cdot
\det\left(e^{(6)}_1, e^{6)}_2,e^{(6)}_3, e^{(6)}_4,e^{(6)}_5, e^{(6)}_{6}\right)\\
=&(-1)^{3\times inv(1,2,3,4)+2\times inv(1,4,2,5,3,6)+2\times inv(1,2,3,4,5,6)}\\
=&1.
\end{align*}

With the help of computer, we find that in $\{\operatorname{val}_{H}(\pi\langle2,2,3\rangle_{I_0})|\pi\in S_{12}\}$, besides 0 there are 182592 `1' and 1152 `-1'. Thus,
$F_{(2,2,3)}[\langle2,2,3\rangle_{I_0}]=\sum_{\pi\in S_{12}} \operatorname{val}_{H}(\pi\langle2,2,3\rangle_{I_0})
=182592-1152=181440$.
Besides $[\langle2,2,3\rangle_{I_0}]$, in $V(\langle2,2,3\rangle,H)$ there are six valid equivalent classes $[\langle2,2,3\rangle_{I_j}]$ ($j=1,2,...,6$). However, in $\{\operatorname{val}_{H}(\pi\langle2,2,3\rangle_{I_j})|\pi\in S_{12}\}$, besides 0 there are 36672 `1' and 36672 `-1'. So we have $F_{(2,2,3)}[\langle2,2,3\rangle_{I_j}]=36672-36672=0$ for $j=1,2,...,6$.
Hence we have
$$F_{(2,2,3)}(\langle2,2,3\rangle)=\sum_{j=0}^6 F_{(2,2,3)}[\langle2,2,3\rangle_{I_j}]=181440.$$
\end{example}
\begin{example}\label{exm-222}
There are 3 valid equivalent classes in $V(\langle2,2,2\rangle,H)$, where $H=B(2,2,2)$. With the help of computer, we find that  $$F_2(\langle2,2,2\rangle)=864.$$
Moreover, we find that
$F_{(1,3,3)}(\langle1,3,3\rangle)=8640$, $F_{(1,4,4)}(\langle1,4,4\rangle)=870912000$ and \\ $F_{(2,2,4)}[\langle2,2,4\rangle_{I_0}]=100362240$.
\end{example}

\begin{remark} Here we make a clearer description of the evaluation $F_{(\ell,m,n)}(\langle \ell,m,n \rangle)$. However, as the direct verification of the Alon-Tarsi conjecture, direct computing the value of $F_{(\ell,m,n)}(\langle \ell,m,n \rangle)$ is only feasible when the number of points of $H=B(n,\ell,m)$ is small. In our personal computer, we can make the evaluation when $|H|\leq16$. The obstacle is that as the growth of $\ell$, $m$, $n$ the time for direct computing  of $F_{(\ell,m,n)}(\langle\ell,m,n\rangle)$ and $F_{(\ell,m,n)}[\langle\ell,m,n\rangle_{I_0}]$ grows rapidly.
For example, if we verify $F_3[\langle3,3,3\rangle_{I_0}]\neq0$, the computer should enumerate 27! times.
\end{remark}

For $\ell$, $m$, $n\in \mathbb{N}$, let $\langle \ell,m,n\rangle_{I_0}$ be the triple labeling defined in Example \ref{ex-typlab}. It is interesting to show the following simple nontrivial problems.
\begin{problem}\label{prob-f3}
\
\begin{enumerate}
  \item Verify that  $F_{3}(\langle3,3,3\rangle)\neq0$ and $F_{3}[\langle3,3,3\rangle_{I_0}]\neq0$;
  \item Let $H=B(3,3,3)$. Find a valid equivalent class
   $[\langle 3,3,3\rangle_I]\in V(\langle 3,3,3\rangle,H)$ such that $F_{3}[\langle3,3,3\rangle_{I}]\neq0$.
\end{enumerate}
\end{problem}

\begin{problem}\label{prob-f22n}
Show that  $F_{(2,2,n)}(\langle2,2,n\rangle)\neq0$ and $F_{(2,2,n)}[\langle2,2,n\rangle_{I_0}]\neq0$.
\end{problem}

Moreover, we have the following conjecture:
\begin{conjecture}\label{conj-flmnmt}
$F_{(\ell,m,n)}(\langle \ell,m,n\rangle)\neq0$ and
$F_{(\ell,m,n)}[\langle \ell,m,n\rangle_{I_0}]\neq0$.
\end{conjecture}

When $\ell=m=n$, Conjecture \ref{conj-flmnmt} and the 3-dimensional Alon-Tarsi problem are closely related.


\subsection{The evaluation of $F_{n}(\langle n^2\rangle)$}\label{subsec-valfnn2}
\

To describe the evaluation $F_{n}(\langle n^2\rangle)$, in \cite[Sect. 5.2]{BI17} B\"{u}rgisser and Ikenmeyer introduced the definition of Latin cube.
In this section, we give a detailed explanation of Latin cube.
In Section \ref{se:clc}, we will find that it also has many interesting combinatorial properties.

\begin{definition}\cite[Definition 5.21]{BI17}\label{def-lcube}
A Latin cube of order $n$ is a map $\alpha:B(n,n,n)\to [n^2]$ that is a bijection on each of the $x$-slices, $y$-slices, and $z$-slices of the combinatorial cube $B(n,n,n)$. The Latin cube is called even if the product of the signs of the resulting permutations of all $x$-slices, $y$-slices, and $z$-slices equals one. Otherwise, the Latin cube is called odd.
\end{definition}

The definition of Latin cube here is a higher-dimensional generalization of the well-known Latin square, see for example \cite{Jan95,DSS12}. However, it is different from the Latin hypercube in \cite{MW08}. Roughly speaking, to obtain a (3-dimensional) Latin cube of order $n$, we can put numbers $i\in[n^2]$ into the points of $B(n,n,n)$ such that
no slice contains any elements more than once. However, unlike Latin squares, the labeling on each point is hard to express on the 2-dimensional paper. When $n=3$, we give an concrete example as follows.
\begin{example}
The obstruction design  $B(3,3,3)$ is given in Figure \ref{fig-333}. Recall that $\textbf{e}_i^{(1)}$ ($i=1,2,3$) denote the $x$-slices, similarly for $y$-slices and $z$-slices.
\begin{figure}[H]
  \centering
  \includegraphics[scale=0.4]{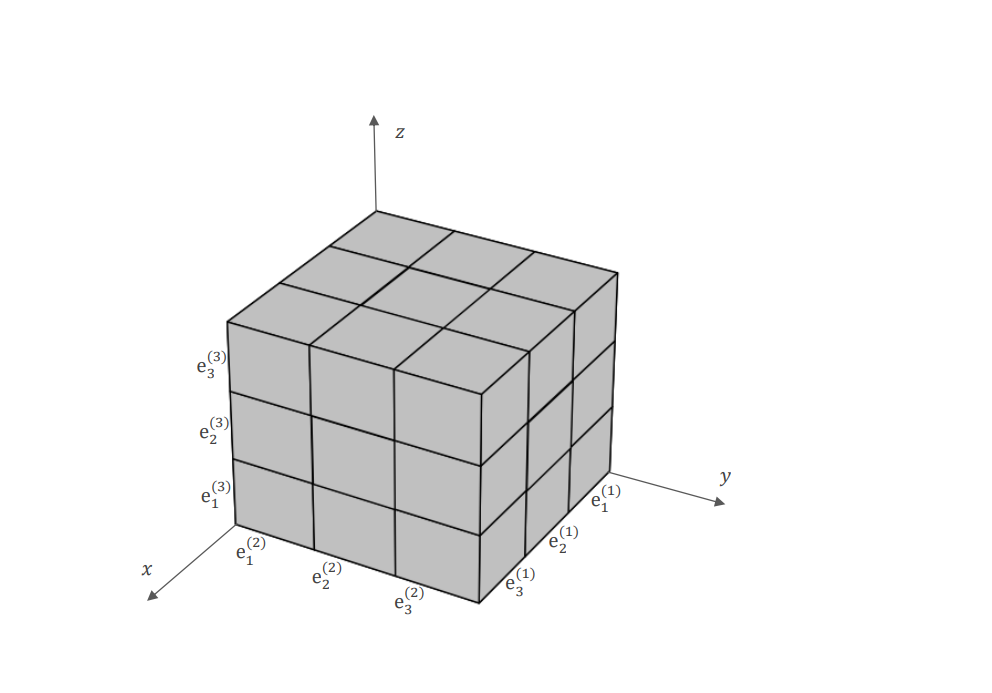}
  \caption{The obstruction design $B(3,3,3)$}\label{fig-333}
\end{figure}

Following the coordinate, for $i=1,2,3$  we can write
$$\textbf{e}_i^{(1)}=\left(
                      \begin{array}{ccc}
                        (i,1,3) & (i,2,3) & (i,3,3)\\
                        (i,1,2) & (i,2,2) & (i,3,2)\\
                        (i,1,1) & (i,2,1) & (i,3,1)\\
                      \end{array}
                    \right).$$
Seeing from the $x$-slices, we define a map $\alpha:B(3,3,3)\to [9]$ as follows:
\begin{align}\label{eq-alx1}
\alpha(\textbf{e}_1^{(1)})=\left(  \begin{array}{ccc}
\alpha( (1,1,3) )& \alpha((1,2,3)) & \alpha((1,3,3))\\
\alpha( (1,1,2) )& \alpha((1,2,2)) & \alpha((1,3,2))\\
\alpha( (1,1,1) )& \alpha((1,2,1)) & \alpha((1,3,1))\\
                      \end{array}
                    \right)
=\left(
\begin{array}{ccc}
     3& 6 & 9 \\
    2 & 5 & 8 \\
    1 & 4 & 7 \\
  \end{array}
\right),
\end{align}
\begin{align}\label{eq-alx2x3}
\alpha(\textbf{e}_2^{(1)})=\left(
\begin{array}{ccc}
    8 & 2 & 5 \\
    7 & 1 & 4 \\
    9 & 3 & 6 \\
  \end{array}
\right),~
\alpha(\textbf{e}_3^{(1)})=\left(
\begin{array}{ccc}
    4 & 7 & 1 \\
    6 & 9 &  3\\
    5 & 8 & 2 \\
  \end{array}
\right).
\end{align}
From (\ref{eq-alx1}) and (\ref{eq-alx2x3}), it is not hard to see that $\alpha(\textbf{e}_i^{(2)})$ and $\alpha(\textbf{e}_i^{(3)})$ exactly
consist of $[9]$. So under the map $\alpha$, $B(3,3,3)$ becomes a
Latin cube of order 3.

To determine the sign of the Latin cube $\alpha$, we can do as follows. Read the entries of $\alpha(\textbf{e}_j^{(k)})$ along the lexicographic order of the points, increasingly. Then each $\alpha(\textbf{e}_j^{(k)})$ becomes a permutation in $S_9$, which is denoted by $\pi(\alpha,\textbf{e}_j^{(k)})$. For example, from (\ref{eq-alx1}) and
(\ref{eq-alx2x3}), we have
\begin{align*}
\pi(\alpha,\textbf{e}_1^{(1)})=\left(
\begin{matrix}
1 & 2 & 3 & 4 & 5 & 6 & 7 & 8 & 9 \\
1 & 2 & 3 & 4 & 5 & 6 & 7 & 8 & 9 \\
\end{matrix}
\right),
\end{align*}
\begin{align*}
\pi(\alpha,\textbf{e}_2^{(1)})=\left(
\begin{matrix}
1 & 2 & 3 & 4 & 5 & 6 & 7 & 8 & 9 \\
9 & 7 & 8 & 3 & 1 & 2 & 6 & 4 & 5 \\
\end{matrix}
\right),
\end{align*}
\begin{align*}
\pi(\alpha,\textbf{e}_3^{(1)})=\left(
\begin{matrix}
1 & 2 & 3 & 4 & 5 & 6 & 7 & 8 & 9 \\
  5  & 6 & 4 & 8 & 9 & 7 & 2 & 3 & 1 \\
\end{matrix}
\right).
\end{align*}
Let $sgn(\pi(\alpha,\textbf{e}_j^{(k)}))$ denote the sign of $\pi(\alpha,\textbf{e}_j^{(k)})$. Then the sign of the Latin cube $\alpha$ is given by
\begin{align}\label{eq-sgn333}
sgn(\alpha)=\prod_{k=1}^3\prod_{j=1}^3 sgn(\pi(\alpha,\textbf{e}_j^{(k)})).
\end{align}

From (\ref{eq-alx1}) and (\ref{eq-alx2x3}), by direct computation we have that $sgn(\alpha)=-1$. So $\alpha$ is an odd Latin cube of order $3$.
\end{example}

Let $\alpha:B(n,n,n)\to [n^2]$ be a Latin cube. As in (\ref{eq-sgn333}), we define the sign of $\alpha$ as
\begin{align}\label{eq-sgnnn}
sgn(\alpha)=\prod_{k=1}^3\prod_{j=1}^n sgn(\pi(\alpha,\textbf{e}_j^{(k)})),
\end{align}
where $sgn(\pi(\alpha,\textbf{e}_j^{(k)}))$ is the sign of the permutation $\pi(\alpha,\textbf{e}_j^{(k)})$ obtained from
$\alpha$ on the slice $\textbf{e}_j^{(k)}$.

Let $\textbf{L}_{n^3}$ denote the set of all Latin cubes of order $n$. Let $L_{n^3}=|\textbf{L}_{n^3}|$ denote the number of all Latin cubes. Let $L^e_{n^3}$ (resp. $L^o_{n^3}$) denote the number of all even (resp. odd) Latin cubes. It was shown in \cite[Prop.5.22]{BI17} that
$$F_{n}(\langle n^2\rangle)=\sum_{\alpha\in \textbf{L}_{n^3}}sgn(\alpha)=L^e_{n^3}-L^o_{n^3},$$
and $F_{n}(\langle n^2\rangle)=0$ when $n$ is odd.
Analogous to the Alon-Tarsi Conjecture, in Problem 5.23 of \cite{BI17}, B\"{u}rgisser and Ikenmeyer raised the following question and we call it 3-dimensional Alon-Tarsi problem.
\begin{problem}[\emph{3-dimensional Alon-Tarsi problem}]\label{prob-lcc}
Let $n$ be even. Does $F_{n}(\langle n^2\rangle)=L^e_{n^3}-L^o_{n^3}\neq0$?
\end{problem}

Many definitions related to Latin square can be extended to Latin cube, for example, the definition of unipotent of \cite{Dan12}.  A Latin cube of order $n$ is called \emph{unipotent} if all the elements of its main diagonal are equal to $n^2$.
Let $U_{n^3}$ denote the number of unipotent Latin cubes. Let $U_{n^3}^e$ (resp. $U_{n^3}^o$) denote the number of even (resp. odd) unipotent Latin cubes.
In Theorem \ref{thm-finvn2-1}, if we let $D=D(n)$ the main diagonal of $B(n,n,n)$, then following the proof of Proposition 5.22 in \cite{BI17} we also have
\begin{proposition}
For $D=D(n)$ and $m=n^2-1$, let $F_{m,D}$ be the invariant defined in (\ref{eq-fm}). Let $\langle m\rangle\in \otimes^3 \mathbb{C}^m$ be the unit tensor. Then
\begin{enumerate}
  \item $F_{m,D}(\langle m\rangle)=0$, when $n$ is odd;
  \item $F_{m,D}(\langle m\rangle)=U_{n^3}^e-U_{n^3}^o$, when $n$ is even.
\end{enumerate}
\end{proposition}
Moreover, it is not hard to see that $F_{m,D}(\langle m\rangle)\neq0$ if and only if  $F_{n}(\langle n^2 \rangle)\neq0$, when $n$ is even.


\section{Some results on the combinatorics of Latin cube}\label{se:clc}

In this section,  results in Section 2 of \cite{DSS12} are generalized to Latin cubes. By the discussions in Section \ref{sec-lmnn2}, we can see that a Latin cube of order $n$ can be considered as an $n\times n\times n$ matrix $L=(L_{ijk})$ whose symbols are from $[n^2]$ such that each symbol occurs exactly once in each $i$-slice, $j$-slice and $k$-slice. By our way of reading entries, we have the following correspondences which determine the sign of $L$.

Each  $i$-slice of $L$ corresponds to a permutation $x_i$ of $[n^2]$ by
$$(j-1)n+k\to L_{ijk},\quad j,~k\in [n].$$
Similarly, each $j$-slice of $L$ corresponds to a permutation $y_j$ of $[n^2]$ by
$$(i-1)n+k\to L_{ijk},\quad i,~k\in [n];$$
each $k$-slice of $L$ corresponds to a permutation $z_k$ of $[n^2]$ by
$$(i-1)n+j\to L_{ijk},\quad i,~j\in [n].$$
The sign of $L$ is
$$sgn(L)=\prod_{i=1}^n sgn(x_i)sgn(y_i)sgn(z_i),$$
where $sgn(\sigma)$ denotes the sign of the permutation $\sigma$.
A Latin cube $L$ is said to be odd if $sgn(L)=-1$ or even if $sgn(L)=+1$. Recall that $L_{n^3}^e$ (resp. $L_{n^3}^o$) denotes
the number of even (resp. odd) Latin cubes of order $n$.

An $n\times n\times n$ (0,1)-matrix $P=(P_{ijk})$ is called a \emph{permutation matrix} if all 1 lie on a diagonal: $\{(i,\sigma_{P}(i),\tau_{P}(i))|i\in [n]\}$ for some $\sigma_{P}$, $\tau_{P}\in S_n$. Then the sign of $P$ is defined as $sgn(P)=sgn(\sigma_{P})sgn(\tau_{P})$. Each symbol $s$ in a Latin cube $L=(L_{ijk})$ determines a permutation matrix $P_s$
where $(P_s)_{ijk} = 1$ whenever $L_{ijk}=s$. Hence $L=\sum_{s=1}^{n^2} sP_s$ and  $\sum_{s=1}^{n^2} P_s$ is the
all-1  matrix.  The symbol sign of $L$ is defined as $sym(L) =
\prod_{s=1}^{n^2} sgn(P_s)$. We say $L$ is symbol-even if $sym(L) = +1$ and symbol-odd if $sym(L)=-1$. Let $L^{se}_{n^3}$ be the number of symbol-even Latin cubes of order $n$ and let  $L^{so}_{n^3}$  be the number of symbol-odd Latin cubes of order $n$.

The following proposition illustrates a difference between
Latin squares and Latin cubes.
\begin{proposition}
Let $L$ be a Latin cube of order $n$. By permuting slices $L$ becomes $L'$. Then $sgn(L)=sgn(L')$.
\end{proposition}
\begin{proof}
By Definition \ref{def-lcube}, $L$ corresponds to a triple labeling $w$ on $H=B(n,n,n)$ and $sgn(L)=\operatorname{val}_H(w)$. By permuting slices, $L$ becomes $L'$ which corresponds to another triple labeling $w'$ and $sgn(L')=\operatorname{val}_H(w')$. So by Lemma \ref{le-signun}, we have $sgn(L)=sgn(L')$.
\end{proof}
So unlike the Latin square of odd order, permuting slices of a Latin cube does not change its parity.

Let $X=(X_{ijk})$ be the $n \times n \times n$ matrix, where each $X_{ijk}$ is a variable. The hyperdeterminant and the hyperpermanent of
$X$ are, respectively, defined by
\begin{align}\label{eq-hydet}
  {\rm Det}(X)=\sum_{\sigma\in S_n}\sum_{\tau\in S_n}sgn(\sigma)sgn(\tau)\prod_{i=1}^n
  X_{i,\sigma(i),\tau(i)},
\end{align}
and
\begin{align*}
{\rm Per}(X)=\sum_{\sigma\in S_n}\sum_{\tau\in S_n}\prod_{i=1}^n
  X_{i,\sigma(i),\tau(i)}.
\end{align*}
The definitions of hyperdeterminant and hyperpermanent here were introduced in \cite{LHR1918}. Many elementary properties were summarized in \cite{Olden1940}. From there we know that (\ref{eq-hydet}) is suitable for the 3-dimensional generalization (see also Remark 2.1 of \cite{AhaMrt15}). Analogous to Theorem 1 of \cite{DSS12}, we have the following theorem.
\begin{theorem}\label{thm-permn2}
$L_{n^3}$ is the coefficient of
$\prod_{i=1}^n\prod_{j=1}^n\prod_{k=1}^n X_{ijk}$ in ${\rm Per}(X)^{n^2}$.
\end{theorem}
\begin{proof}
Let $PM(n)$ be the set of $n\times n\times n$ permutation matrices. For $P\in PM(n)$, let $\sigma_P,\tau_P\in S_n$ be the permutation pair defined by $P$. Hence
\begin{align}\label{eq-permn2}
({\rm Per}(X))^{n^2}=\left(\sum_{P\in PM(n)}\prod_{i=1}^n
  X_{i,\sigma_P(i),\tau_P(i)}\right)^{n^2}=
  \sum_{\mathcal{P}\in PM(n)^{\times n^2}}\prod_{P_s\in \mathcal{P}}\prod_{i=1}^{n}X_{i,\sigma_{P_s}(i),\tau_{P_s}(i)}.
\end{align}
$\prod_{P_s\in \mathcal{P}}\prod_{i=1}^{n}X_{i,\sigma_{P_s}(i),\tau_{P_s}(i)}$ is squarefree if and only if when $\mathcal{P}\in\{(P_1,P_2,...,P_{n^2})\in PM(n)^{\times n^2}$\\ $| \sum_{1\leq s \leq n^2}P_s~\text{is the all-1 matrix}\}=\{(P_1,P_2,...,P_{n^2})\in PM(n)^{\times n^2}|\sum_{s=1}^{n^2}sP_s~\text{is a Latin cube}\}$.
The result follows since $\prod_{i=1}^n\prod_{j=1}^n\prod_{k=1}^n X_{ijk}$ is the only squarefree product of $n^3$ elements in $\{X_{ijk}\}_{i,j,k\in[n]}$.
\end{proof}
For symbol-even and symbol-odd Latin cubes, we have the following theorem.
\begin{theorem}\label{thm-seso}
$L^{se}_{n^3}-L^{so}_{n^3}$ is the coefficient of $\prod_{i=1}^n\prod_{j=1}^n\prod_{k=1}^n X_{ijk}$ in ${\rm Det}(X)^{n^2}$.
\end{theorem}
\begin{proof}
The proof is similar to that of Theorem \ref{thm-permn2}, except (\ref{eq-permn2}) is replaced by
\begin{align}\label{eq-detxn2}
({\rm Det}(X))^{n^2}=
  \sum_{\mathcal{P}\in PM(n)^{\times n^2}}\prod_{P_s\in \mathcal{P}} sgn(\sigma_{P_s})sgn(\tau_{P_s}) \prod_{i=1}^{n}X_{i,\sigma_{P_s}(i),\tau_{P_s}(i)}.
\end{align}

The squarefree terms again arise precisely when $L:=\sum_{1\leq s\leq n^2} sP_s$ is a Latin cube, but now we also multiply by $\prod_{P_s\in \mathcal{P}} sgn(\sigma_{P_s})sgn(\tau_{P_s})$, which is the symbol-sign of $L$.
\end{proof}

Let $B_n$ be the set of $n\times n\times n$ (0,1)-matrices. For $A\in B_n$, let $\sigma_0(A)$ be the number of 0 entries in $A$.
\begin{theorem}\label{thm-detAn2}
$$L^{se}_{n^3}-L^{so}_{n^3}=\sum_{A\in B_n}(-1)^{\sigma_0(A)}{\rm Det}(A)^{n^2}.$$
\end{theorem}
\begin{proof}
Let $PM(n)$ be the set of $n\times n\times n$ permutation matrices. Then (\ref{eq-detxn2}) implies
\begin{align}\label{eq-Alelo}
\sum_{A\in B_n}(-1)^{\sigma_0(A)}{\rm Det}(A)^{n^2}=
\sum_{(A,\mathcal{P})\in B_{n}\times (PM(n))^{\times n^2}}Z(A,\mathcal{P}),
\end{align}
where
$$Z(A,\mathcal{P})=(-1)^{\sigma_0(A)}\prod_{P_s\in \mathcal{P}}sgn(P_s)\prod_{i=1}^n A_{i,\alpha_{P_s}(i),\beta_{P_s}(i)},$$
and $\alpha_{P_s}$, $\beta_{P_s}$ are the permutations of $[n]$ defined by $P_s$ for $1\leq s\leq n^2$.

For most $(A,\mathcal{P})\in B_n\times (PM(n))^{\times n^2}$, there exist $i,j,k\in [n]$ for which $(P_s)_{ijk}=0$ for all $1\leq s \leq n^2$. If this occurs, we can define $A^c$ to be the matrix formed by toggling $A_{ijk}$ in the lexicographically first such coordinate $ijk$, that is, $A^c_{ijk}=1$ and others entries of $A^c$ are the same as $A$. Hence $Z(A,P)=-Z(A^c,P)$ and we can pair up and cancel all such terms in the sum. The right hand side of (\ref{eq-Alelo}) simplifies to
\begin{align}\label{eq-Alelo+}
\sum_{\mathcal{P}\in (PM(n))^*}\prod_{s=1}^{n^2}sgn(P_s)=L_{n^3}^{se}-L_{n^3}^{so},
\end{align}
where $(PM(n))^*:=\{(P_1,P_2,...,P_{n^2})\in (PM(n))^{\times n^2}|\sum_{s=1}^{n^2} sP_s~\text{is a Latin cube}\}$, and $A$ must be the all-1 matrix.
\end{proof}
Just as the discussion in Section 2 of \cite{DSS12}, the value of $L^{se}_{n^3}-L^{so}_{n^3}$ could also be extracted from ${\rm Det}(X)^{n^2}$ by differentiation as follows
$$L^{se}_{n^3}-L^{so}_{n^3}=\frac{\partial}{\partial x_{111}}\bigg|_{x_{111}=0}~\frac{\partial}{\partial x_{112}}\bigg|_{x_{112}=0}\cdots~\frac{\partial}{\partial x_{nnn}}\bigg|_{x_{nnn}=0}~{\rm Det}(X)^{n^2}.$$

\begin{theorem}
$$L_{n^3}=\sum_{A\in B_n}(-1)^{\sigma_0(A)}{\rm Per}(A)^{n^2}.$$
\end{theorem}
\begin{proof}
By definition we have $L_{n^3}=L^{se}_{n^3}+L^{so}_{n^3}$.
In the proof of Theorem \ref{thm-detAn2}, if we replace the
hyperdeterminant in (\ref{eq-Alelo}) with hyperpermanent, then
(\ref{eq-Alelo+}) is replaced by
\begin{align*}
\sum_{\mathcal{P}\in (PM(n))^*}\prod_{s=1}^{n^2}1=L_{n^3}^{se}+L_{n^3}^{so},
\end{align*}
where $(PM(n))^*:=\{(P_1,P_2,...,P_{n^2})\in (PM(n))^{\times n^2}|\sum_{s=1}^{n^2} sP_s~\text{is a Latin cube}\}$.
\end{proof}
In Theorem \ref{thm-detAn2}, if $A'$ is obtained from $A$ by permuting two  $j$-slices of $A$ (or $k$-slices), then by Theorem 2 of \cite{Olden1940} we have ${\rm Det}(A)=-{\rm Det}(A')$. So combining Proposition 5.22 of \cite{BI17} we have the following.
\begin{proposition}
When $n$ is odd, $L^{se}_{n^3}-L^{so}_{n^3}=L^{e}_{n^3}-L^{o}_{n^3}=0$.
\end{proposition}

Equation (1.1) of \cite{DSS12} provides the relation between symbol-parities and parities of Latin squares (see also the discussion in \cite{Jan95}). Analogously, we have the following problem.
\begin{problem}
When $n$ is even, what is the relation between $L^{se}_{n^3}-L^{so}_{n^3}$ and
$L^{e}_{n^3}-L^{o}_{n^3}$?
\end{problem}
In Section 3 of \cite{DSS12}, another proof of Alon-Tarsi conjecture for odd prime $p$ was given.
\begin{problem}
How can we generalize the discussions in Section 3 of \cite{DSS12} to the case of 3-dimensional Alon-Tarsi Problem? That is,  show that $F_{n}(\langle n^2\rangle)\neq0$ when $n=p-1$ for all odd prime $p$.
\end{problem}

\section{On the generalizations to even and odd dimensions}
\label{se-evenodd}

The \emph{order} of a tensor is the number of its dimensions (ways, modes). So vectors are tensors of order 1 and matrices are
tensors of order 2.
In previous sections, we use 3-dimensional obstruction designs $B(n,n,n)\setminus D$ and $B(\ell,m,n)$ to describe the invariants for tensors of order 3. Although, it is straightforward to generalize the results in previous sections to other dimensions. However, there are some distinctions between even and odd dimensions.
More precisely, in Lemma 7.2.7 of  \cite{Iken12}, by the language of graph theory, Ikenmeyer pointed out a forbidden pattern for 3-dimensional obstruction designs, which states that an obstruction design should contain no two vertices lying in the same three hyperedges.
This forbidden pattern (see Definition 7.2.9 of \cite{Iken12}) leads to the equivalent definition of obstruction designs used in this paper. However, from the proof of Lemma 7.2.7 of  \cite{Iken12}, it is not hard to see that the forbidden pattern is only necessary for the generalization to odd dimensional obstruction designs. In the following, we will give some examples to illustrate the distinction.


\subsection{$B(n,n)$ as 2-dimensional obstruction design}
\

The discussions in Section \ref{se:prelim} can be generalized to
tensors of order 2 in a straightforward way. We make a concise discussion here. Note that our discussion is different from Example 4.10 of \cite{BI13}.

Suppose that  $H=B(n,n)=\{(i,j)|i\in [n],j\in[n]\}\subseteq \mathbb{Z}_{+}^2$.
Consider the slices that parallel to the coordinate axis
$$\textbf{e}_i^{(1)}=\{(x,y)\in H|x=i\},\quad 1\leq i\leq n,$$
$$\textbf{e}_j^{(2)}=\{(x,y)\in H|y=j\},\quad 1\leq j\leq n.$$
The set
$$E^{(1)}=\{\textbf{e}_i^{(1)}|1\leq i\leq n\},$$
consisting of the $x$-slices of $H$ defines a set partition of $H$.
The $x$-marginal distribution of $H$ is: $m^{(1)}=(|\textbf{e}_1^{(1)}|,|\textbf{e}_2^{(1)}|,...,|\textbf{e}_n^{(1)}|)=(n,n,...,n)$. Similarly, we define the set partition $E^{(2)}$ of $y$-slices of $H$ with its marginal distribution $m^{(2)}=(n,n,...n)$. Just as the discussion in Section \ref{se:prelim}, $H$ is a 2-dimensional obstruction design of type $(\lambda,\mu)$, where $\lambda=\mu=n\times n$. Moreover, we have $|\textbf{e}_i^{(1)} \cap \textbf{e}_j^{(2)} |\leq1$.

If we put the lexicographic ordering on $H=B(n,n)$, then $H\simeq [n^2]$. On $x$-slices (resp. $y$-slices),  assume that $\textbf{e}_i^{(1)}<\textbf{e}_j^{(1)}$ (resp. $\textbf{e}_i^{(2)}<\textbf{e}_j^{(2)}$) for $i<j$. Then
$E^{(1)}=\{\textbf{e}_i^{(1)}|i\in [n]\}$ and $E^{(2)}=\{\textbf{e}_i^{(2)}|i\in [n]\}$ are ordered set partitions of $H$. So there exist $\pi^{(1)}_H=\epsilon$, $\pi^{(2)}_H \in S_{n^2}$, such that
$$F_{(n,n)}:=F_H=\widehat{\langle\lambda,\mu|}(\epsilon,\pi_{H}^{(2)})
\mathcal{P}_{n^2}\in \operatorname{Sym}^{n^2}(\otimes^2\mathbb{C}^n)^*.$$
$F_{(n,n)}$ can be considered as a homogeneous polynomial in $\mathcal{O}\left(\otimes^2\mathbb{C}^{n}\right)_{n^2}$, whose evaluation on $\omega\in \otimes^2\mathbb{C}^{n}$ is given by
\begin{align}\label{eq-fhn2}
F_{(n,n)}(\omega)&=\operatorname{val}_{H}(\omega^{\otimes n^2})=
\widehat{\langle\lambda,\mu|}(\epsilon,\pi_{H}^{(2)})
\mathcal{P}_{n^2}|\omega^{\otimes n^2}\rangle,
\end{align}
where $\operatorname{val}_{H}$ is defined as in (\ref{eq-valh}) (just by omitting the evaluation on $E^{(3)}$).

Let $\langle n \rangle=\sum_{i=1}^n e_i\otimes e_i$ be the unit tensor of
$\otimes^2\mathbb{C}^n$. As the definitions in \cite{Jan95} and \cite{DSS12}, let $L_n^{E}$ (resp. $L_n^{O}$) be the number of even (resp. odd) Latin squares of order $n$. Then it is not hard to get the following proposition.

\begin{proposition}\label{prop-fnn}
By (\ref{eq-fhn2}) we have
$F_{(n,n)} ((g_1\otimes g_2)w)=\det(g_1)^{n}\det(g_2)^{n} F_{(n,n)}(w)$
for all $g_1,~g_2\in \operatorname{GL}_{n}$ and $w\in \mathbb{C}^{n} \otimes \mathbb{C}^{n}$. Moreover,
$$F_{(n,n)}(\langle n \rangle)=L_n^{E}-L_n^{O}.$$
\end{proposition}

\begin{corollary}\label{cor-fnnoe}
If $n$ is odd, then $F_{(n,n)}=0$. When $n$ is even,  $F_{(n,n)}\neq0$ if and only if the Alon-Tarsi Conjecture is true.
\end{corollary}
\begin{proof}
(1) By Proposition \ref{prop-fnn}, if $n$ is odd, then it is well known that $F_{(n,n)}(\langle n \rangle)=L_n^{E}-L_n^{O}=0$. So we have
$F_{(n,n)}(g_1\otimes g_2 \cdot \langle n \rangle)=0$ for any $g_1$, $g_2\in \operatorname{GL}_n$.
By equation (1) of \cite[Sect. 1]{OR20}, we know that the orbit
$\{g_1\otimes g_2\cdot \langle n \rangle|~g_1, g_2\in \operatorname{GL}_n\}$ is Zariski-dense in
$\otimes^2\mathbb{C}^n$. So we have $F_{(n,n)}=0$.

(2) Suppose that $n$ is even. If the Alon-Tarsi Conjecture is true, then
by Proposition \ref{prop-fnn} we have $F_{(n,n)}(\langle n \rangle)=L_n^{E}-L_n^{O}\neq0$.
On the other hand, if $F_{(n,n)}\neq0$, then for any Zariski-dense set there should exist a point which $F_{(n,n)}$ doesn't vanish on it. Specially, let $S=\{g_1\otimes g_2\cdot \langle n \rangle|~g_1, g_2\in \operatorname{GL}_n\}$ be the orbit of $\operatorname{GL}_n^2$. Then there exists $g\otimes h \cdot \langle n \rangle\in S$ such that $F_{(n,n)}(g\otimes h \cdot \langle n \rangle)\neq0$. Since $F_{(n,n)}(g\otimes h\cdot \langle n \rangle)=\det(g)^{n}\det(h)^{n} F_{(n,n)}(\langle n \rangle)$, we have
$F_{(n,n)}(\langle n \rangle)=L_n^{E}-L_n^{O}\neq0$. That is, the Alon-Tarsi Conjecture is true.
\end{proof}

However, by the well-known Cauchy's formula \cite[(6.3.2)]{Procesi07}, as $\operatorname{GL}_n^2$-representations we have
\begin{align*}
\operatorname{Sym}^{n^2}(\mathbb{C}^{n} \otimes \mathbb{C}^{n})\simeq\bigoplus_{\lambda\vdash n^2, \ell(\lambda)\leq n} \{\lambda\}\otimes \{\lambda\}.
\end{align*}
Therefore, no matter $n$ is even or odd, we always have $\{n\times n\}\otimes \{n\times n\}\in \operatorname{Sym}^{n^2}(\mathbb{C}^{n} \otimes \mathbb{C}^{n})$.
This implies that no matter $n$ is even or odd, there always exists a nonvanishing invariant $F\in O(\mathbb{C}^{n} \otimes \mathbb{C}^{n})^{\operatorname{SL}_n^2}_{n^2}$ such that $F((g_1\otimes g_2)w)=\det(g_1)^{n}\det(g_2)^{n} F_{(n,n)}(w)$
for all $g_1,~g_2\in \operatorname{GL}_{n}$ and $w\in \mathbb{C}^{n} \otimes \mathbb{C}^{n}$. Therefore, when $n$ is odd, $B(n,n)$ is not suitable for the construction of the invariant $F\in O(\mathbb{C}^{n} \otimes \mathbb{C}^{n})^{\operatorname{SL}_n^2}_{n^2}$.

\subsection{$B(n_1,n_2,...,n_k)$ as $k$-dimensional obstruction design}
\

In this part, by the discussion of Kronecker coefficients, for higher dimensional generalizations, we continue to see the distinctions.

Firstly, equation (4.4.13) of \cite{Iken12} can be generalized as follows.  Suppose that $n_1$, $n_2$,...$n_k\in \mathbb{N}$ and $k\geq 3$. Let $N=n_1n_2\cdots n_k$ be their products. The 1-dimensional rectangular irreducible $\operatorname{GL}_{n_1n_2\cdots n_{k-1}}$-representation $\{\frac{N}{n_k}\times n_k\}$, which corresponds to the $n_k$th power of the determinant, decomposes as follows:
\begin{align}\label{eq-glm1mk}
\left\{\frac{N}{n_k}\times n_k\right\}\bigg\downarrow^{\operatorname{GL}_{n_1n_2\cdots n_{k-1}}}_{\operatorname{GL}_{n_1}\times \operatorname{GL}_{n_2}\cdots\times \operatorname{GL}_{n_{k-1}}}
=\left\{n_1\times\frac{N}{n_1}\right\}\otimes\left\{n_2\times \frac{N}{n_2}\right\}\cdots\otimes\left\{n_{k-1}
\times\frac{N}{n_{k-1}}\right\}
\end{align}
If $k$ is odd (resp. even) , then $k-1$ is even (resp. odd). Since the Kronecker coefficient is invariant when two of its partitions taking transpose (see e. g. Lemma 4.4.7 of \cite{Iken12}). By (\ref{eq-glm1mk}) above, we have the following proposition.
\begin{proposition}\label{prop-kn1n2nk}
Suppose that $n_1$, $n_2$,...,$n_k\in \mathbb{N}$ and $k\geq3$. Let $N=n_1n_2\cdots n_k$ and $\lambda^{(i)}=\frac{N}{n_i}\times n_i$ be partitions of $N$. Then we have
\begin{enumerate}
  \item if $k$ is odd, then $k(\lambda^{(1)},\lambda^{(2)},...,\lambda^{(k)})=1$;
  \item if $k$ is even, then $k({}^t\lambda^{(1)},\lambda^{(2)},...,\lambda^{(k)})=1$, where ${}^t\lambda^{(1)}$ is the conjugate of $\lambda^{(1)}$.
\end{enumerate}
\end{proposition}
In Proposition \ref{prop-kn1n2nk}, if $n_1=n_2=\cdots=n_k=n$, then we have the following.
\begin{proposition}\label{prop-kkroncube}
Suppose that $n_1=n_2\cdots=n_k=n$ and $k\geq3$. Let $\lambda^{(1)}=\lambda^{(2)}=\cdots=\lambda^{(k)}=n^{k-1}\times n$ be partitions of $n^k$. Then we have
\begin{enumerate}
  \item\label{it-kodd} if $k$ is odd, then $k(\lambda^{(1)},\lambda^{(2)},...,\lambda^{(k)})=1$;
  \item\label{it-keven} if $k$ is even, then $k(\lambda^{(1)},\lambda^{(2)},...,\lambda^{(k)})>1$.
\end{enumerate}
\end{proposition}
\begin{proof}
(1) follows from Proposition \ref{prop-kn1n2nk}.
Suppose that $k$ is even. Let $\chi_{\lambda^{(i)}}$ denote character of the irreducible representation of $S_{n^k}$ determined by $\lambda^{(i)}=n^{k-1}\times n$.
Set $\lambda^{(1)}=\lambda^{(2)}=...=\lambda^{(k)}=\lambda$. Then it is well-known that
\begin{align*}
  k(\lambda^{(1)},\lambda^{(2)},...,\lambda^{(k)})
&=\frac{1}{n^k!}\sum_{\sigma\in S_{n^k}}\chi_{\lambda^{(1)}}(\sigma)\chi_{\lambda^{(2)}}(\sigma)\cdots
\chi_{\lambda^{(k)}}(\sigma) \\
&= \frac{1}{n^k!}\sum_{\sigma\in S_{n^k}}\chi_{\lambda}(\sigma)^k
>\frac{1}{n^k!}\sum_{\sigma\in S_{n^k}}\chi_{\lambda}(\sigma)^2\\
&=1,
\end{align*}
where the last equality follows from the orthogonality relation of
irreducible characters (see e. g. \cite[(2.10)]{Fulton91}).
\end{proof}
Let $B^{(k)}(n)=\{(i_1,i_2,...,i_k)|i_1,i_2,...,i_k\in [n]\}\subseteq \mathbb{Z}^{k}_+$ denote the $k$-dimensional hypercube of side length $n$. The forbidden pattern of Lemma 7.2.7 of \cite{Iken12} can be generalized to arbitrary odd number $k$, straightforwardly. That is, the intersection of  hyperedges chosen from each set partition of the obstruction design contains vertices at most once. By Theorem 1 and 2 of \cite{Fish1991}, we know that $B^{(k)}(n)$ is the unique set in  $\mathbb{Z}^{k}_+$ whose margins consist of $n\times n^{k-1}$. This implies that $B^{(k)}(n)$ is the unique obstruction design with type $(n^{k-1}\times n,n^{k-1}\times n,...,n^{k-1}\times n)$. So by Proposition \ref{prop-kkroncube}, when $k$ is odd, $B^{(k)}(n)$ is the unique obstruction design to describe the evaluation of the invariant of $\otimes^{k}\mathbb{C}^{n^{k-1}}$ that corresponds to the Kronecker coefficient $k(n^{k-1}\times n,n^{k-1}\times n,...,n^{k-1}\times n)$.

On the other hand, when $k$ is even, from (\ref{it-keven}) of Proposition \ref{prop-kkroncube}, there are more than one obstruction designs  to describe the evaluation of the invariants of $\otimes^{k}\mathbb{C}^{n^{k-1}}$ that correspond to the Kronecker coefficient $k(n^{k-1}\times n,n^{k-1}\times n,...,n^{k-1}\times n)$.
So there always exists some obstruction design that doesn't satisfy the forbidden pattern. Hence $B^{(k)}(n)$ is not the unique obstruction design with type $(n^{k-1}\times n,n^{k-1}\times n,...,n^{k-1}\times n)$. Moreover, from Corollary \ref{cor-fnnoe} we can see that $B^{(k)}(n)$ may not be used as the obstruction design. However, just like Problem 5.23 of \cite{BI17}, the Alon-Tarsi conjecture can also be generalized to $B^{(k)}(n)$. So if both $k$ and $n$ are even and the $k$-dimensional Alon-Tarsi conjecture is hold for $B^{(k)}(n)$, then $B^{(k)}(n)$ can be used to describe the evaluation of the invariant of $\otimes^{k}\mathbb{C}^{n^{k-1}}$.

\section{Final remarks and open problems}\label{se:final}



In this section, besides the questions mentioned in previous sections, we summarize some other related questions. We hope that all these questions can be attacked by the interested readers in related fields.

\subsection{}
In \cite{BI17}, B\"{u}rgisser and Ikenmeyer discussed the stabilizer period, stabilizer and polystability of the matrix multiplication tensor $\langle n,n,n \rangle$ and the unit tensor $\langle m\rangle$. Recently, in \cite{LimKY20}, the authors showed that the asymptotic exponent of matrix multiplication also equals to asymptotic
exponent for the product operation in Lie algebras, Jordan algebras, and
Clifford algebras. So it is interesting to determine the stabilizer period, stabilizer and polystability of the structure tensors of bilinear maps discussed in \cite{LimKY20}. Just as the discussion in Section 5.1 of \cite{BI17}, we also want to know the minimal generic fundamental invariants of these tensors. In particular, let $T_{\mathfrak{sl}_n}$ denote the structure tensor of $\mathfrak{sl}_n$. From \cite{Bari21}, we know that $T_{\mathfrak{sl}_n}\in \mathbb{C}^{n^2-1}\otimes \mathbb{C}^{n^2-1}\otimes \mathbb{C}^{n^2-1}$ whose minimal generic fundamental invariant is given in Section \ref{se:od3}.

\subsection{}
Let $m=n^2-2$. When $n$ is even, (\ref{it-it2}) of Theorem \ref{thm-deln} implies that, up to a scaling factor, there is exactly one homogeneous $\operatorname{SL}^3_m$-invariant $F_m: \otimes^3 \mathbb{C}^m\to \mathbb{C}$ of degree $n^3-2n$ (and no nonzero invariant of smaller degree). As the discussion in Section \ref{se:od3}, the construction of $F_m$ should be obtained by deleting two diagonals of $B(n,n,n)$. So we want to know how to determine these diagonals.

Generally, suppose that $(n-1)^2<m<n^2$ and $m=n^2-k$ for some $k$.
Then to construct an $\operatorname{SL}^3_m$-invariant $F_m: \otimes^3 \mathbb{C}^m\to \mathbb{C}$ of degree $n^3-kn$, we can delete $k$ diagonals of $B(n,n,n)$. The difficulty is to determine which $k$ diagonals should be chosen such that $F_m\neq0$. Let $\langle m \rangle=\sum_{i=1}^{m} |iii\rangle\in \otimes^3 \mathbb{C}^m$ be the unit tensor. The evaluation $F_m(\langle m \rangle)$ is also interesting. It can be considered as a generalized 3-dimensional Alon-Tarsi problem.
Another interesting problem is to classify equivalent obstruction designs when $k$ diagonals of $B(n,n,n)$ are deleted. This implies to classify 3-dimensional (0,1)-matrices $(x_{ijk})_{1\leq i,j,k\leq n}$ with constant margin $m=n^2-k$. The enumeration of 3-dimensional (0,1)-matrices was discussed in \cite{Barv17}. In \cite{PPV20}, the authors present both upper and lower bounds for the Kronecker coefficients and the reduced Kronecker coefficients, in term of the number of certain contingency tables.

\subsection{}
Consider a group $G$ that acts by linear transformations on the complex vector space $V$. Following the notations in \cite{Derk01}, define
$\beta_{G}(V)=\min\{d:\mathcal{O}(V)^G~\text{is generated by}$
$\text{invariants of degree}\leq d\}$.  The lower and upper bounds for $\beta_{G}(V)$ are discussed in \cite{Derkmak20} and \cite{Derk01}, respectively. In particular, setting $G=\operatorname{SL}_{n_1}\times \operatorname{SL}_{n_2} \times \operatorname{SL}_{n_3}$ and $V=\mathbb{C}^{n_1}\otimes\mathbb{C}^{n_2}\otimes \mathbb{C}^{n_3}$, can we describe the estimation of the bounds of $\beta_{G}(V)$ more clearly by the results of \cite{Derk01} and \cite{Derkmak20}? The application of $\beta_{G}(V)$ was studied in \cite{BFG+19}.


\subsection{}
To determine the degree $d$ of the nonvanishing invariants of
$\mathcal{O}\left(\mathbb{C}^{m_1}\otimes \mathbb{C}^{m_2}\otimes \mathbb{C}^{m_3}\right)$ under the action of $\operatorname{SL}_{m_1}\times \operatorname{SL}_{m_2} \times \operatorname{SL}_{m_3}$, by (\ref{eq-d1d2d3}) it is equivalent to determine the set
$$\Delta(m_1,m_2,m_3)=\{(\delta_1,\delta_2,\delta_3)\in \mathbb{Z}_+^3~|~
k(m_1\times\delta_1,m_2\times\delta_2,m_3\times\delta_3)>0\}.$$
In particular, suppose that $m_1=m_2=m_3=m$. Then $\delta_1=\delta_2=\delta_3$ which is denoted by $\delta$. As in \cite{BI17}, let $k_m(\delta)$ denote $k(m\times\delta, m\times\delta, m\times\delta)$.
Recall that $\delta(m)$ is the minimal integer satisfying $k_m(\delta(m))>0$. Then Problem 5.19 of \cite{BI17} asks if $k_m(\delta)>0$ for all $\delta\geq\delta(m)$ and $m>2$. The asymptotic growth of $k_3(\delta)$ was discussed in \cite{Bald18}. It is also well known that $\Delta(m_1,m_2,m_3)$ is a finitely generated semigroup.
Recently, many theories have been developed to determine the positivity of Kronecker coefficients \cite{Bald18,Fei19,Geloun21,Res20}.
So using these results, can we give a characterization of the set $\Delta(m_1,m_2,m_3)$?

\subsection{}

Inspired by Kruskal's classical results \cite{Krus77} (see also \cite[Sect. 12.5]{Lan12}), we introduce a class of tensors $\{V_r|r\in \mathbb{N}\}\subseteq\mathbb{C}^{\ell m}\otimes\mathbb{C}^{mn}\otimes\mathbb{C}^{n\ell}$ which could be used to verify $F_{(\ell,m,n)}\neq0$ directly.
They are given as follows.

Let $i=1,2,...,r$ and $x_i,~y_i,~z_i\in \mathbb{C}$.
Define $\textbf{a}_i\in\mathbb{C}^{\ell m}$, $\textbf{b}_i\in\mathbb{C}^{ mn}$ and $\textbf{c}_i\in\mathbb{C}^{n\ell}$ by
$$\textbf{a}_i=(1,x_i,x_i^2,...,x_i^{\ell m-1})^t,\quad\textbf{b}_i=(1,y_i,y_i^2,...,y_i^{ mn-1})^t,\quad\textbf{c}_i=(1,z_i,z_i^2,...,z_i^{ n\ell-1})^t,$$
where `$t$' denotes the transposition. Let
$$V_r=\sum_{i=1}^r \textbf{a}_i\otimes\textbf{b}_i\otimes\textbf{c}_i.$$
The definition of $V_r$ is a 3-dimensional generalization of the  Vandermonde matrix, that is, the matrix corresponds to the Vandermonde determinant. Let $R(V_r)$ be the tensor rank  of $V_r$. Suppose that $x_i\neq x_j$, $y_i\neq y_j$ and $z_i\neq z_j$ for $i\neq j$. Let $A_r=(\textbf{a}_1,\textbf{a}_2,\cdots,\textbf{a}_r)$, $B_r=(\textbf{b}_1,\textbf{b}_2,\cdots,\textbf{b}_r)$ and $C_r=(\textbf{c}_1,\textbf{c}_2,\cdots,\textbf{c}_r)$. If $r\geq \ell m$, then every $\ell m$ columns of $A_r$ are linear independent. If $r<\ell m$, then all $r$ columns of $A_r$ are linear independent. Similar results hold for $B_r$ and $C_r$. Hence, if $r\leq \frac{1}{2}(\ell m+mn+n\ell)-1$, then by \cite[Thm. 12.5.3.1]{Lan12}  we have $R(V_r)=r$. We have the following questions.
\begin{itemize}
  \item Suppose that $x_i\neq x_j$, $y_i\neq y_j$ and $z_i\neq z_j$ for $i\neq j$. What is the value of $R(V_r)$ when $r>\frac{1}{2}(\ell m+mn+n\ell)-1$? Does $R(V_r)$ increase as the increasing of $r$? In particular, suppose that $x_i=y_i=z_i=i$ are integers for $i=1,2,...,r$. What is the value of $R(V_r)$ for each $r\in \mathbb{N}$?
  \item Suppose that $x_i=y_i=z_i=i$ for $i=1,2,...,r$. When $r\geq \max\{\ell^2,m^2,n^2\}$, it is easy to construct valid equivalent class for $\operatorname{val}_H$ from $V_r^{\otimes\ell mn}$, where $H=B(n,\ell,m)$. So does $F_{(\ell,m, n)}(V_r)\neq0$ when $r\geq \max\{\ell^2,m^2,n^2\}$?
\end{itemize}

\subsection{}
In \cite{Col84}, it was shown that the problem of completing partial Latin squares is NP-complete. So what is the complexity of completing Latin cubes discussed in Section \ref{se:clc}?

\subsection{}
The conjectures of Section 1 of \cite{HRota94} consist of the Dinitz Conjecture, Alon-Tarsi Conjecture and Rota's basis Conjecture, etc. They all have $k$-dimensional generalizations.
For example, we make a 3-dimensional generalization of these conjectures as follows.

A \emph{partial Latin cube} of order $n$ is an $n\times n\times n$ array of symbols with the property that no symbol appears more
than once in any slices.

\begin{conjecture}[3-dimensional Dinitz Conjecture]
Associate to each triple $(i,j,k)$ where $1\leq i,j,k\leq n$ a set $S_{ijk}$ of size $n^2$. Then there exists a partial Latin cube $(a_{ijk})_{1\leq i,j,k\leq n}$ with $a_{ijk}\in S_{ijk}$ for all triples $(i,j,k)$.
\end{conjecture}

So a Latin cube $(a_{ijk})_{l\leq i,j,k\leq n}$
is a partial Latin cube such that all sets $S_{ijk}$ are identical with the set $\{1,2,...,n^2\}$. Therefore, the 3-dimensional Alon-Tarsi Conjecture is equivalent to Problem \ref{prob-lcc} of Section \ref{sec-lmnn2}.

\begin{conjecture}[3-dimensional Rota Conjecture]
Let $V$ be a vector space over an arbitrary infinite field and
$\dim V=n^2$. Suppose $B_1$,$B_2$,...,$B_n$ are $n$ sets of bases of $V$. Then for each $i$, there is a linear order
of $B_i$, say $B_1=\{b_{1jk}|j,k\in [n]\}$, $B_2=\{b_{2jk}|j,k\in [n]\}$,...,$B_n=\{b_{njk}|j,k\in [n]\}$ such that the following two sets
\begin{enumerate}
\item $C_1=\{b_{i1k}|i,k\in [n]\}$, $C_2=\{b_{i2k}|i,k\in [n]\}$,...,$C_n=\{b_{ink}|i,k\in [n]\}$,
\item $D_1=\{b_{ij1}|i,j\in [n]\}$, $D_2=\{b_{ij2}|i,j\in [n]\}$,...,$D_n=\{b_{ijn}|i,j\in [n]\}$
\end{enumerate}
are $n$ sets of bases, respectively.
\end{conjecture}

As in \cite{HRota94},  what's the relation between these conjectures?
The Dinitz Conjecture was solved by Galvin \cite{Gal95}. Recently,
Pokrovskiy showed that Rota's Basis Conjecture holds asymptotically \cite{Pok20}.

\section*{Acknowledgments}
We express our appreciation to the referees and editors for their helpful suggestions in improving the manuscript. We are grateful to our colleagues and friends in ZJUT for their help and encouragement. We are also grateful to Songling Shan of Illinois State University for so many fruitful discussions. Special thanks to Fedja Nazarov for the helpful discussions when we preparing this paper. When writing this paper, we try to follow the nice tips of \cite{P18}. Xin Li is grateful to Yaqiong Zhang for her support and understanding.

\bibliographystyle{amsplain}

\begin{thebibliography}{20}

\bibitem{AhaMrt15}
Ron Aharoni, Martin Loebl, \emph{The odd case of Rota's bases conjecture}, Adv. Math. 282(2015), 427-442.

\bibitem{AJR13+}
E. Allman, P. Jarvis, J. Rhodes, J. Sumner,
\emph{Tensor rank, invariants, inequalities, and applications},
SIAM J. Matrix Anal. Appl. 34(2013), no.3, 1014-1045.

\bibitem{AT92}
N. Alon, M. Tarsi, \emph{Colorings and orientations of graphs},
Combinatorica 12(1992), no.2, 125-134.

\bibitem{AY22}
A. Amanov, D. Yeliussizov, \textit{Fundamental invariants of tensors, Latin hypercubes, and rectangular Kronecker coefficients},
International Mathematics Research Notices, 20(2023), 17552-17599.

\bibitem{Bald18}
V. Baldoni, M. Vergne,
\emph{Computation of dilated Kronecker coefficients. With an appendix by M. Walter.}
J. Symbolic Comput. 84(2018), 113-146.

\bibitem{Bari21}
Kashif Bari, \textit{On the structure tensor of $\mathfrak{sl}_n$},
Linear Algebra Appl. 653(2022), 266-286.

\bibitem{Barv17}
Alexander Barvinok,
\emph{Counting integer points in higher-dimensional polytopes},  Convexity and concentration, 585-612, IMA Vol. Math. Appl., 161, Springer, New York, 2017.


\bibitem{Bessen04}
C. Bessenrodt, C. Behns, \textit{On the Durfee size of Kronecker products of characters of the symmetric
group and its double covers}, Journal of Algebra, 280(2004) 132-144.


\bibitem{BDI20}
M. Bl\"{a}ser, J. D\"{o}rfler, C. Ikenmeyer,
\emph{On the complexity of evaluating highest weight vectors}, Proceedings of the 36th Computational Complexity Conference (CCC 2021), LIPIcs 200, 29:1-29:36, 2021.


\bibitem{BlaI18}
M. Bl\"{a}ser, C. Ikenmeyer, \textit{Introduction to geometric complexity theory}, preprint, 2018, \url{https://www.dcs.warwick.ac.uk/~u2270030/teaching_sb/summer17/introtogct/gct.pdf}.



\bibitem{BCS97}
P. B\"{u}rgisser, M. Clausen, and M. Amin Shokrollahi. \emph{Algebraic complexity theory}, volume 315 of Grundlehren der Mathematischen Wissenschaften, Springer-Verlag, Berlin, 1997.

\bibitem{BGO+17}
P. B\"{u}rgisser, A. Garg, R. Oliveira, M. Walter, and Avi Wigderson,  \emph{Alternating minimization, scaling algorithms, and the null-cone problem from invariant theory}, In
Proceedings of Innovations in Theoretical Computer Science (ITCS 2018). arXiv:1711.08039, 2017.


\bibitem{BFG+19}
P. B\"{u}rgisser, C. Franks, A. Garg, R. Oliveira, M. Walter, Avi Wigderson, \emph{Towards a theory of non-commutative optimization: geodesic first and second order methods for moment maps and polytopes}, arXiv:1910.12375v3, 2019.


\bibitem{BI11}
P. B\"{u}rgisser, C. Ikenmeyer, \emph{Geometric
complexity theory and tensor rank}, Proceedings 43rd
Annual ACM Symposium on Theory of Computing
2011, pp.509-518.


\bibitem{BI13}
P. B\"{u}rgisser, C. Ikenmeyer, \textit{Explicit lower bounds via geometric complexity theory},  Proceedings 45rd Annual ACM Symposium on Theory of Computing 2013, pp.141-150.


\bibitem{BI17}
P. B\"{u}rgisser, C. Ikenmeyer, \textit{Fundamental invariants of orbit closures}, Journal of Algebra,  477(2017), 390-434.

\bibitem{BLMW11}
P. B\"{u}rgisser, J. M. Landsberg, L. Manivel, J. Weyman,
\textit{An overview of mathematical issues arising in the geometric complexity theory approach to $VP\neq NVP$}, SIAM J. Comput. 40(2011), no.4, 1179-1209.


\bibitem{BLM+21}
P. B\"{u}rgisser,  M. Levent Do\u{g}an, V. Makam, M. Walter, Avi Wigderson, \emph{Polynomial time algorithms in invariant theory for torus actions},  arXiv:2102.07727v1, 2021.

\bibitem{CHLan19}
A. Conner, A. Harper, J. M. Landsberg,
\textit{New lower bounds for matrix multiplication and $\det_3$}, Forum Math. Pi 11(2023), Paper No. e17, 30 pp.

\bibitem{Col84}
Charles J. Colbourn,
\emph{The complexity of completing partial Latin squares},
Discrete Appl. Math. 8(1984), no.1, 25-30.

\bibitem{Dan12}
Kotlar Daniel, \textit{On extensions of the Alon-Tarsi Latin square conjecture}, Electron. J. Combin. 19(2012), no.4, Paper7, 10pp.

\bibitem{Derk01}
Harm Derksen, \emph{Polynomial bounds for rings of invariants,}
Proc. Amer. Math. Soc. 129(2001), no.4, 955-963.

\bibitem{Derkmak20}
Harm Derksen, Visu Makam,
\emph{An exponential lower bound for the degrees of invariants of cubic forms and tensor actions}, Adv. Math. 368(2020), 107136, 25 pp.

\bibitem{Fei19}
Jiarui Fei,
\emph{Cluster algebras, invariant theory, and Kronecker coefficients \uppercase\expandafter{\romannumeral2}},
Adv. Math. 341(2019), 536-582.


\bibitem{Fish1991}
P. C. Fishburn, J. C. Lagarias, J. A. Reeds, and L. A. Shepp, \emph{Sets uniquely determined by projections on axes \uppercase\expandafter{\romannumeral2} Discrete case}, Discrete
Math. 91(1991), 149-159.


\bibitem{Fulton91}
W. Fulton, J. Harris, \emph{Representation theory. A first course.} Graduate Texts in Mathematics, 129. Readings in Mathematics. Springer-Verlag, New York, 1991.

\bibitem{Gal95}
Fred  Galvin,
\emph{The list chromatic index of a bipartite multigraph},
J. Combin. Theory Ser. B, 63(1995), no.1, 153-158.

\bibitem{GIM+20}
A. Garg, C. Ikenmeyer, V. Makam, R. Oliveira,  M. Walter, Avi Wigderson,
\emph{Search problems in algebraic complexity, GCT, and hardness of generators for invariant rings}, 35th Computational Complexity Conference, Art. 12, 16 pp.,
LIPIcs. Leibniz Int. Proc. Inform., 169, Schloss Dagstuhl. Leibniz-Zent. Inform., Wadern, 2020.

\bibitem{Geloun21}
Joseph B. Geloun, Sanjaye Ramgoolam,
\emph{Quantum mechanics of bipartite ribbon graphs: Integrality, Lattices and Kronecker coefficients},
Algebr. Comb. 6(2023), no.2, 547-594.


\bibitem{HRota94}
R. Huang and G.-C. Rota, \emph{On the relations of various conjectures on Latin squares and straightening coefficients}, Discrete Math., 128(1994) 225-236.


\bibitem{Iken12}
C. Ikenmeyer, \textit{Geometric Complexity Theory, Tensor Rank, and Littlewood-Richardson Coefficients}, PhD thesis, Institute of Mathematics, University of Paderborn, 2012.


\bibitem{Jan95}
J. C. M. Janssen, \emph{On even and odd Latin squares}, J. Combin. Theory Ser. A, 69(1995), pp. 173-181.


\bibitem{Krus77}
Joseph B. Kruskal, \emph{Three-way arrays: rank and uniqueness of trilinear decompositions, with application to arithmetic complexity and statistics,} Linear Algebra Appl. 18(1977), no.2, 95-138.


\bibitem{Lan12}
J. M. Landsberg. \emph{Tensors: geometry and applications}, volume 128 of Graduate Studies in Mathematics. American Mathematical Society, Providence, RI, 2012.

\bibitem{Lan17}
J. M. Landsberg, \emph{Geometry and complexity theory}, volume 169 of Cambridge Studies in Advanced Mathematics. Cambridge University Press, Cambridge, 2017.

\bibitem{LimKY20}
Lek-Heng Lim and Ke Ye, \emph{Ubiquity of the exponent of matrix multiplication}, In Proceedings of the 45th International Symposium on Symbolic and Algebraic Computation, ISSAC '20, page 8-11, New York, NY, USA, 2020. Association for Computing Machinery.


\bibitem{Maniv10}
L. Manivel,
\emph{A note on certain Kronecker coefficients,} Proc. Amer. Math. Soc. 138(2010), no.1, 1-7.

\bibitem{MW08}
Brendan D. McKay, Ian M. Wanless,
\emph{A census of small Latin hypercubes}, SIAM J. Discrete Math. 22(2008), no. 2, 719-736.


\bibitem{OR20}
G. Ottaviani, P. Reichenbach, \textit{Tensor Rank and Complexity}, arXiv:2004.01492, 2020.

\bibitem{Olden1940}
Rufus Oldenburger, \emph{Higher dimensional determinants},
Amer. Math. Monthly 47(1940), 25-33.


\bibitem{PPV16}
I. Pak, G. Panova, E. Vallejo, \textit{Kronecker products, characters, partitions, and the tensor square conjectures},
Adv. Math., 288(2016), 702-731.

\bibitem{P18}
I. Pak, \emph{How to Write a Clear Math Paper: Some 21st Century Tips}, Journal of Humanistic Mathematics, Volume 8 Issue 1, pages 301-328, 2018.

\bibitem{PPV20}
I. Pak, G. Panova, \textit{Bounds on Kronecker coefficients via contingency tables}, Linear Algebra Appl., 602(2020) 157-178.

\bibitem{Pok20}
Alexey Pokrovskiy,
\emph{Rota's Basis Conjecture holds asymptotically},
arXiv:2008.06045v1, 2020.

\bibitem{Procesi07}
Claudio Procesi, \emph{Lie groups. An approach through invariants and representations}, Universitext. Springer, New York, 2007.

\bibitem{Res20}
Nicolas Ressayre,
\emph{Vanishing symmetric Kronecker coefficients},
Beitr. Algebra Geom., 61(2020), no.2, 231-246.



\bibitem{LHR1918}
L. H. Rice, \emph{$P$-Way Determinants, with an Application to Transvectants}, Amer. J. Math., 40(1918), no.3, 242-262.

\bibitem{DSS12}
 D. S. Stones,
\emph{Formulae for the Alon-Tarsi conjecture}, SIAM J. Discrete Math., 26(2012), no.1, 65-70.

\bibitem{Tewari}
V. V. Tewari, \textit{Kronecker coefficients for some near-rectangular partitions},
Journal of Algebra,  429(2015), 287-317.
\end{thebibliography}

\end{document}